\documentclass[10pt,a4paper]{amsart}

\usepackage{amssymb, mathrsfs, amscd}
\usepackage[all]{xy}

\usepackage{graphics}
\usepackage{graphicx}
\usepackage{enumerate}
\usepackage{color}
\usepackage{a4wide}

\newtheorem{theorem}{Theorem}[section]
\newtheorem{definition}[theorem]{Definition}
\newtheorem{lemma}[theorem]{Lemma}

\newtheorem{proposition}[theorem]{Proposition}

\newtheorem{remark}[theorem]{Remark}

\title{Stable finiteness properties of infinite discrete groups}

\author{No\'{e} B\'{a}rcenas}
\address{Centro de Ciencias Matem{\'a}ticas, UNAM, Campus Morelia, Michoac{\'a}n, Mexico }
\email{barcenas@matmor.unam.mx}

\author{Dieter Degrijse}
\address{Department of Mathematical Sciences, Univeristy of Copenhagen, Denmark}
\email{d.degrijse@math.ku.dk}

\author{Irakli Patchkoria}
\address{Department of Mathematical Sciences, Univeristy of Copenhagen, Denmark}
\email{irakli.p@math.ku.dk}

\newcommand{\mF}{\mathcal {F}}

\newcommand{\Z}{\mathbb Z}

\newcommand{\orb}{\mathcal{O}_{\mF}G}
\newcommand{\orbmod}{\mbox{Mod-}\mathcal{O}_{\mF}G}
\newcommand{\zmod}{\mathbb{Z}\mbox{-Mod}}

\newcommand{\Gspec}{\mathrm{Sp}_{G}}
\newcommand{\Hspec}{\mathrm{Sp}_{H}}
\newcommand{\Kspec}{\mathrm{Sp}_{K}}
\newcommand{\colim}{\mathrm{colim} }
\newcommand{\ul}{\underline }
\medskip

\begin{document}

\begin{abstract}
Let $G$ be an infinite discrete group. A classifying space for proper actions of $G$ is a proper $G$-CW-complex $X$ such that the fixed point sets $X^H$ are contractible for all finite subgroups $H$ of $G$. In this paper we consider the stable analogue of the classifying space for proper actions in the category of proper $G$-spectra and study its finiteness properties. We investigate when $G$ admits a stable classifying space for proper actions that is finite or of finite type and relate these conditions to the compactness of the sphere spectrum in the homotopy category of proper $G$-spectra and to classical finiteness properties of the Weyl groups of finite subgroups of $G$. Finally, if the  group $G$ is virtually torsion-free we also show that the smallest possible dimension of a stable classifying space for proper actions coincides with the virtual cohomological dimension of $G$, thus providing the first geometric interpretation of the virtual cohomological dimension of a group. 
\end{abstract}
\maketitle
\section{Introduction}
Let $G$ be an infinite discrete group and let $\mathcal{F}$ be the family of finite subgroups of $G$. Recall that a classifying space for proper actions of $G$ is a proper $G$-CW complex $X$ such that the fixed point sets $X^H$ are contractible for every finite subgroup $H$ of $G$. Such a space is also called a model for $\underline{E}G$. 
These spaces appear naturally in geometric group theory and algebraic topology, and are important tools for studying groups (e.g. see \cite{Luck2}). Classifying spaces for proper actions also have important $K$-theoretical applications. They appear for example on left hand side of assembly map conjectures (e.g.~see \cite{LuckReich}) and in a generalization of the Atiyah-Segal completion theorem (see \cite{LuckOliver}). With these applications and others in mind it is important to have models for $\underline{E}G$ with good geometric finiteness properties. There has therefore been a longstanding interest in finiteness properties of classifying spaces for proper actions, e.g.~see \cite{BradyLearyNucinkis}, \cite{KrophollerMislin}, \cite{LearyNucinkis}, \cite{Luck1}, \cite{nucinkis00} to list just a few references.

One can give an interpretation of $\underline{E}G$ in terms of homotopical algebra. There is a model category structure on the category of pointed $G$-spaces, where the weak equivalences are those maps which induce weak homotopy equivalences on fixed points for all finite subgroups of $G$ (see \cite[Section 2.3]{Fau}). This model category is known as the \emph{unstable proper $G$-equivariant homotopy theory}. By adding a disjoint $G$-fixed basepoint to a model for $\underline{E}G$, one obtains a cofibrant replacement of $S^{0}$ in the latter model structure. One can think of this cofibrant replacement as a proper cellular decomposition of $S^{0}$ up to homotopy. In this paper we consider the analogue of this in the \emph{stable model category of proper $G$-spectra} $\Gspec$ and define a stable model for $\underline{E}G$ to be a certain proper cellular decomposition of the sphere spectrum $S^0$ in $\Gspec$ (see below for the precise definition). Here, a $G$-spectrum is just an orthogonal spectrum equipped with a $G$-action and the stable model category of proper $G$-spectra $\Gspec$ and its associated homotopy category $\mathrm{Ho}(\Gspec)$ are defined in \cite{Fau}. The weak equivalences in $\Gspec$ are morphisms of $G$-spectra which induce stable equivalences on derived fixed points for all finite subgroups. In Section 3 we will briefly recall the definition of a $G$-spectrum and discuss some of the properties of $\Gspec$ and $\mathrm{Ho}(\Gspec)$ that will be needed in this paper. Most of these are contained in \cite[Section 6]{Fau}. 

The forthcoming paper \cite{DHLPS} provides an alternative model structure on $\Gspec$ which is Quillen equivalent to the one given in \cite{Fau}. The weak equivalences in both model structures are the same. However, the model structure of \cite{DHLPS} has much better control of fibrations compared to that of \cite{Fau} which is obtained by an abstract Bousfield localization procedure. The paper \cite{DHLPS} also conducts a detailed study of the proper $G$-equivariant stable homotopy category $\mathrm{Ho}(\Gspec)$ and provides deloopings with respect to equivariant vector bundles. Furthermore, several equivariant cohomology theories, such as Borel cohomology, Bredon cohomology, equivariant $K$-theory and equivariant stable cohomotopy are shown to be represented in $\mathrm{Ho}(\Gspec)$ (see \cite{DHLPS}). This tells us that the category $\mathrm{Ho}(\Gspec)$ contains interesting information about the group $G$ and understanding its properties is therefore useful. For example, if $G$ is an infinite discrete group, then the sphere spectrum $S^0 \in \mathrm{Ho}(\Gspec)$ is in general not a compact object. Hence, it makes sense to study the finiteness properties of $S^0 \in \mathrm{Ho}(\Gspec)$, i.e.~ \emph{the stable finiteness properties of $G$}. The purpose of this paper is to conduct such a study.

\begin{definition}  \rm A \emph{stable model for} $\underline{E}G$ consists of a collection of $G$-spectra $\{X^n\}_{n\in \mathbb{Z}}$ together with a collection of morphisms $ X^{n-1}\rightarrow  X^n$ in $\mathrm{Ho}(\Gspec)$, where $ X^{n}=\{\ast\}$ if $n<0$ while for each $n\geq 0$ there exists a stable cofiber sequence
\[      X^{n-1} \rightarrow  X^{n} \rightarrow   \bigvee_{i \in I_n} \Sigma^n G/H_{i \ +} \rightarrow \Sigma  X^{n-1}   \]
such that $H_i \in \mathcal{F}$ for all $i \in I_n$, and 
\[ \mathrm{hocolim}_n  X^n \cong S^0\]
in $\mathrm{Ho}(\Gspec)$. We will often refer to a specific model for $\mathrm{hocolim}_n  X^n$ as a stable model for $\underline{E}G$, keeping in mind that the specific $G$-spectra $X^n$ and morphisms $ X^{n-1}\rightarrow  X^n$ are part of the data. If there exists a $d$ such that $X^{n-1}\xrightarrow{\cong} X^{n}$ for all $n\geq d+1$, then the stable model for $\underline{E}G$ is called \emph{finite dimensional}. In this case, the smallest such $d$ is called the \emph{dimension} of the stable model for $\underline{E}G$. A stable model for $\underline{E}G$ is said to be of \emph{finite type} if the sets $I_n$ are finite for all $n\geq 0$. A stable model for $\underline{E}G$ is a said to be \emph{finite} if it is both finite dimensional and of finite type. 
\end{definition}

Recall that the geometric dimension for proper actions $\underline{\mathrm{gd}}(G)$ of $G$ is the smallest possible dimension that a model for $\underline{E}G$ can have. This geometric invariant coincides with its algebraic counterpart, called the Bredon cohomological dimension $\underline{\mathrm{cd}}(G)$, except for the possibility that $\underline{\mathrm{cd}}(G)=2$ but $\underline{\mathrm{gd}}(G)=3$ (see Section 4). We define \emph{the stable geometric dimension for proper actions of $G$}, denoted by $\underline{\mathrm{gd}}_{\mathrm{st}}(G)$, to be the smallest possible dimension that a stable model for $\underline{E}G$ can have. One of the main results of this paper (see Theorem \ref{th: main theorem}) says that the geometric invariant $\underline{\mathrm{gd}}_{\mathrm{st}}(G)$ equals the Mackey cohomological dimension $ \underline{\mathrm{cd}}_{\mathcal{M}}(G) $ of $G$. This algebraic invariant was introduced in \cite{MartinezNucinkis06} using Mackey functors for infinite discrete groups and shown to be equal to the virtual cohomological dimension of $G$ when $G$ is virtually torsion-free. We will recall some basics about the category of $G$-Mackey functors $\mathrm{Mack}_{\mathcal{F}}G$ in Section 2.
\begin{theorem}For any discrete group $G$, one has
	\[    \underline{\mathrm{cd}}_{\mathcal{M}}(G)  =  \underline{\mathrm{gd}}_{\mathrm{st}}(G).    \]
	In particular, if $G$ is virtually torsion-free then
	$   \mathrm{vcd}(G) =  \underline{\mathrm{gd}}_{\mathrm{st}}(G)$.    
\end{theorem}
\noindent The virtual cohomological dimension $\mathrm{vcd}(G)$ of a virtually torsion-free group $G$ is by definition the cohomological dimension of any finite index torsion-free subgroup of $G$. It is a classical algebraic invariant that satisfies $\mathrm{vcd}(G)\leq \underline{\mathrm{cd}}(G)$. But since this inequality can be strict, the theorem above provides the first known geometric interpretation of the virtual cohomological dimension of a virtually torsion-free group. \\

We also investigate when $G$ admits a stable model for $\underline{E}G$ of finite type (see Theorem \ref{th: finite type}, Theorem \ref{prop: compact} and Theorem \ref{th: finite}) and relate this condition to the compactness of the sphere spectrum in $\mathrm{Ho}(\Gspec)$ and to classical finiteness properties of the Weyl-groups $W_G(H)=N_G(H)/H$ of finite subgroups $H$ of $G$.
\begin{theorem}\label{th: intro2} For any countable group $G$, the following are equivalent.
\begin{itemize}
\item[-] The sphere spectrum $S^{0}$ is a compact object of  $\mathrm{Ho}(\Gspec)$.
\item[-] There exists a finite dimensional stable model for $\underline{E}G$ and there exists a finite type stable model for $\underline{E}G$.
\item[-]  There exists a finite length resolution of the Burnside ring functor $\underline{A}$ in $\mathrm{Mack}_{\mathcal{F}}G$ consisting of finitely generated projective modules.
\item[-]  The Mackey cohomological dimension $ \underline{\mathrm{cd}}_{\mathcal{M}}(G)$ is finite, there are only finitely many conjugacy classes of finite subgroups in $G$, and for every finite subgroup $H$ of $G$ there exists a resolution of $\mathbb{Z}$ in $\mathbb{Z}[W_G(H)]\mbox{-mod}$ consisting of finitely generated projective modules. 
\end{itemize}
	
\end{theorem}

In Section \ref{sec: suspension} we show that the suspension $\Sigma^{\infty}X_{+}$ of a proper $G$-CW-complex $X$ gives rise to a stable model for $\underline{E}G$ if and only if $X^H$ is acyclic for every $H \in \mathcal{F}$. This indicates that the class of stable models for $\underline{E}G$ contains more than just the suspensions of certain proper $G$-CW-complexes and hence truly is a richer class. This fact is illustrated in the last section of the paper where we consider an example given by Leary and Petrosyan in \cite{LP} of a group $G$ that does not admit a $2$-dimensional contractible proper $G$-CW-complex but satisfies $\mathrm{vcd}(G)=2$ and $\underline{\mathrm{cd}}(G)=\underline{\mathrm{gd}}(G)=3$. For this group we explicitly construct a $2$-dimensional stable model for $\underline{E}G$. The reader will see that the construction crucially involves the attaching of equivariant cells via transfer maps that are unavailable in the unstable setting. \\

\noindent \textbf{Ackknowledgements.} We would like to thank Markus Hausmann, Wolfgang L\"{u}ck and Stefan Schwede for useful conversations. The first author was supported by DGAPA-PAPIIT Grant IA100315. The second and third author were supported by the Danish National Research Foundation through the Centre for Symmetry and Deformation (DNRF92) .

\section{Bredon modules and Mackey functors}
Throughout this section, let $G$ be a discrete group and let $\mathcal{F}$ be the family of finite subgroups of $G$. 
Let us begin by recalling some basics concerning modules over the orbit category and Bredon cohomology. This cohomology theory was introduced by Bredon in \cite{Bredon} for finite groups 
as a means to develop an obstruction theory for equivariant extensions of maps. It was later generalized to arbitrary groups by L\"{u}ck with applications to finiteness conditions  (see \cite[Section 9]{Luck} and \cite{LuckMeintrup}). We refer the reader to \cite[Section 9]{Luck} for more details. 

The \emph{orbit category} $\orb$ is the category with $G$-sets $G/H$, $H\in\mathcal{F}$, as objects and $G$-maps as morphisms. Note that the set of morphisms $\mathrm{Mor}(G/H,G/K)$ can be identified with the fixed point set $(G/K)^H$ and that a $G$-map $G/H \rightarrow G/K : H \mapsto xK$ will be denoted by $G/H \xrightarrow{x} G/K$. A right $\orb$-module is a contravariant functor \[M: \orb \rightarrow \zmod.\] The right $\orb$-modules are the objects of an abelian category $\orbmod$, whose morphisms are natural transformations. The abelian group of morphisms  between two objects $M,N\in\orbmod$ is denoted by $\mathrm{Hom}_{\orb}(M,N)$. Similarly, one defines the category of 
left $\orb$-modules whose objects are covariant functors from the $\orb$ to abelian groups.

The category $\orbmod$ has free objects; more precisely, its free objects are isomorphic to direct sums of basic free modules, which are of the form
$$\Z[-,G/K]$$
where $K\in\mathcal{F}$ and $\Z[G/H,G/K]$ is the free $\Z$-module with basis the set of $G$-maps from $G/H$ to $G/K$. A free module is finitely generated if it is isomorphic to a finite direct sum of basic free modules. The module $\Z[-,G/K]$ is free in the sense that there is a Yoneda lemma which states that for any $M \in \orbmod$, one has a natural isomorphism
\[  \mathrm{Hom}_{\orb}(\Z[-,G/K],M) \xrightarrow{\cong} M(G/K): \varphi \mapsto \varphi(G/K)(\mathrm{Id}_{G/K}).       \]
Similarly, one also has a natural isomorphism
\begin{equation}\label{eq: tensor with free} \Z[-,G/K] \otimes_{\orb} N \xrightarrow{\cong} N(G/K): \varphi\otimes m \mapsto N(\varphi)(m)       \end{equation}
for any $K \in \mathcal{F}$ and any left $\orb$-module $N$. Here  $-\otimes_{\orb}-$ denotes the categorical tensor product over the orbit category. The analogous statements for covariant free modules of course also hold true.

A sequence of modules in $\orbmod$ is said to be exact if it is exact when evaluated at every object. A right $\orb$-module $P$ is called projective if the functor 
\[ \mathrm{Hom}_{\orb}(P,-): \orbmod \rightarrow \zmod: M \mapsto  \mathrm{Hom}_{\orb}(P,M) \]
is exact. By the Yoneda lemma, free modules are projective and each module admits a free (projective) resolution. An $\orb$-module $M$ is said to be finitely generated if there exists a surjection of a finitely generated free module onto $M$.

The $n$-th Bredon cohomology of $G$ with coefficients in a right $\orb$-module $M$ is denoted by $\mathrm{H}^n_{\mathcal{F}}(G,M)$
and is defined as the $n$-th cohomology of the cochain complex obtained by applying the contravariant functor $\text{Hom}_{\orb}(-,M)$ to a free (projective) resolution of the constant functor $\underline{\Z}:G/H\mapsto\Z$ that sends objects to $\mathbb{Z}$ and all morphism to the identity map. In other words, one has \[\mathrm{H}^{\ast}_{\mathcal{F}}(G,M)=\mathrm{Ext}_{\orb}^{\ast}(\underline{\mathbb{Z}},M).\] The \emph{Bredon cohomological dimension} of $G$ is defined as
\[ \underline{\mathrm{cd}}(G) = \sup\{ n \in \mathbb{N} \ | \ \exists M \in \orbmod :  \mathrm{H}^n_{\mathcal{F}}(G,M)\neq 0 \}. \]
Standard techniques in homological algebra show that the number $\underline{\mathrm{cd}}(G)$ coincides with the length of the shortest free (or projective) resolution of $\underline{\mathbb{Z}}$ in $\orbmod$ (e.g.~see \cite[Lemma 4.1.6]{weibel}).\\

We now turn to Mackey functors. Mackey functors were introduced for finite groups by Dress and Green, and were studied extensively in this context by Th{\'e}venaz, Webb and others (e.g.~see \cite{ThevenazWebb95}, \cite{Webb00}). However, many of the elementary results obtained about Mackey functors for finite groups generalize to infinite groups and their family of finite subgroups. Our treatment of cohomology of Mackey functors follows the approach of Mart{\'\i}nez-P{\'e}rez and Nucinkis in \cite{MartinezNucinkis06}. We will briefly recall some facts about Mackey functors here and refer to \cite{Degrijse}, \cite{MartinezNucinkis06}, \cite{ThevenazWebb95} and \cite{Webb00} for more details.

Consider the diagrams of the form
\begin{equation} \label{eq: mackey morphism}G/S \xleftarrow{\varphi} \Delta \xrightarrow{\psi} G/K \end{equation}
where the maps $\varphi$ and $\psi$ are $G$-equivariant, $S,K \in \mathcal{F}$ and $\Delta = \coprod_{i=1}^n G/H_i$ is a finitely generated $G$-set with stabilizers in $\mathcal{F}$. A diagram of the form $(\ref{eq: mackey morphism})$ is said to be equivalent to a diagram \[G/S \xleftarrow{\varphi'} \Delta' \xrightarrow{\psi'} G/K\] if there exists a $G$-equivariant bijection $\theta: \Delta \rightarrow \Delta'$ such that $\varphi' \circ \theta =\varphi $ and $\psi' \circ \theta =\psi $. The set of equivalence classes of diagrams of the form $(\ref{eq: mackey morphism})$ is denoted by $\omega_{\mathcal{F}}(S,K)$. Note that we also allow the empty morphism $G/S \leftarrow \emptyset \rightarrow G/K$. One can check that $\omega_{\mathcal{F}}(S,K)$ is a free abelian monoid with disjoint union of $G$-sets as addition and the empty morphism as neutral element. The \emph{Mackey category} (or Burnside category) $\mathcal{M_F}G$ is defined as follows. Its objects are the $G$-sets $G/H$ for all $H \in \mathcal{F}$. The space of morphisms $\mathrm{Mor}(G/S,G/K)$ is by definition the abelian group completion of $\omega_{\mathcal{F}}(S,K)$. Composition is defined by taking pullbacks in the category of $G$-sets on basis morphisms and then extended by linearity.
We will denote the set of morphisms from $G/S$ to $G/K$ in $\mathcal{M_F}G$ by $\mathbb{Z}^G[S,K]$. Let us point out that this group is different from $\mathbb{Z}[G/S,G/K]$, which is the free abelian group generated by the morphisms from $G/S$ to $G/K$ in the orbit category $\orb$. The category $\mathrm{Mack}_{\mathcal{F}}G$ is the category with objects the contravariant additive functors $M: \mathcal{M_F}G \rightarrow \zmod$, and morphisms all natural transformations between these functors. An object of $\mathrm{Mack}_{\mathcal{F}}G$ is called a \emph{Mackey functor}. The category $\mathrm{Mack}_{\mathcal{F}}G$ is again an abelian category with enough projectives in which one can do homological algebra in a similar way as in $\orbmod$.  The free Mackey functors are the Mackey functors that are isomorphic to direct sums of functors of the form $\mathbb{Z}^G[-,K]$, for $K \in \mathcal{F}$. The \emph{Mackey cohomological dimension} of $G$, denoted by $\underline{\mathrm{cd}}_{\mathcal{M}}(G)$, is by definition
\[ \underline{\mathrm{cd}}_{\mathcal{M}}(G)=\sup\{ n \in \mathbb{N} \ | \ \exists M \in \mathrm{Mack}_{\mathcal{F}}G : \mathrm{Ext}^{n}_{\mathcal{M}_{\mathcal{F}}G}(\underline{A},M) \neq 0 \}. \]
Here $\underline{A}$ is the Burnside ring functor that takes $H \in \mathcal{F}$ to the Burnside ring $A(H)$ of $H$. As before, one shows using standard techniques that the invariant $\underline{\mathrm{cd}}_{\mathcal{M}}(G)$ coincides with the length of the shortest free (or projective) resolution of the Burnside ring functor in $ \mathrm{Mack}_{\mathcal{F}}G$. \\ 

Functors between orbit categories and Mackey categories give rise to induction, coinduction and restriction functors on the level of module categories, satisfying the usual adjointness properties (e.g. see \cite[Section 2]{MartinezNucinkis06}). One can construct a functor 

\begin{equation} \label{eq: pi functor}
\pi_G : \orb \rightarrow \mathcal{M_F}G
\end{equation}
that maps an object $G/K$ in $\orb$ to the object $G/K$ in $\mathcal{M_F}G$ and takes a morphism $G/S \xrightarrow{x} G/K$ to the morphism in the Mackey category represented by $G/S \xleftarrow{\mathrm{Id}} G/S \xrightarrow{x} G/K$. When the group under consideration is clear from the context, we will simply denote $\pi_G$ by $\pi$. The associated restriction and induction functors 
 \begin{equation}\label{eq: res2} \mathrm{res}_{\pi}: \mathrm{Mack}_{\mathcal{F}}G \rightarrow \orbmod: M \mapsto M\circ \pi = M^{\ast}, \end{equation}and
\begin{equation}\label{eq: ind} \mathrm{ind}_{\pi}: \orbmod \rightarrow \mathrm{Mack}_{\mathcal{F}}G : N \mapsto N(?) \otimes_{\orb}\mathbb{Z}^G[-,\pi(?)] \end{equation}
are connected via the adjointness isomorphism
\begin{equation}\label{eq: important adjointness} \mathrm{Hom}_{\orb}(N,M^{\ast}) \cong \mathrm{Hom}_{\mathcal{M_F}G}(\mathrm{ind}_{\pi}(N),M) \end{equation}
and the natural isomorphism
\begin{equation}   \label{eq: tensor adj} \mathrm{ind}_{\pi}(N)\otimes_{\mathcal{M}_{\mathcal{F}}G} L \cong N \otimes_{\orb}\mathrm{res}_{\pi}(L).   \end{equation}
where $M,N$ are contravariant functors and $L$ is a covariant functor. Similar constructions and adjointness isomorphisms of course also hold true when considering covariant additive functors $\mathcal{M_F}G \rightarrow \zmod$.

One can show that $ \mathrm{ind}_{\pi}(\mathbb{Z}[-,G/H])=\mathbb{Z}^G[-,H]$ and $ \mathrm{ind}_{\pi}(\underline{\mathbb{Z}})=\underline{A}$ (see \cite[Th. 3.7]{MartinezNucinkis06}). The functor $ \mathrm{ind}_{\pi}$ is not exact in general, but does preserve exactess of projective resolutions, which yields that
(see \cite[Th. 3.8.]{MartinezNucinkis06})  
\[\mathrm{Ext}^{n}_{\mathcal{M}_{\mathcal{F}}G}(\underline{A},M)\cong \mathrm{Ext}^{n}_{\mathcal{O}_{\mathcal{F}}G}(\underline{\mathbb{Z}},M^{\ast})=\mathrm{H}^{n}_{\mathcal{F}}(G,M^{\ast}) .\] 
for every $M \in  \mathrm{Mack}_{\mathcal{F}}G $ and every $n \in \mathbb{N}$. This implies that 
\begin{equation} \label{eq: cohom dim}  \underline{\mathrm{cd}}_{\mathcal{M}}(G)\leq \underline{\mathrm{cd}}(G). \end{equation}
If $G$ is virtually torsion-free,  its \emph{virtual cohomological dimension}, denoted by $\mathrm{vcd}(G)$, is by definition the cohomological dimension of any torsion-free finite index subgroup of $G$. This notion is well-defined due to a result of Serre (e.g. see \cite[Ch. VIII]{brown}). A surprising result proven by Mart{\'\i}nez-P{\'e}rez and Nucinkis in \cite{MartinezNucinkis06} says that, if $G$ is virtually torsion-free one has 
\[    \mathrm{vcd}(G)= \underline{\mathrm{cd}}_{\mathcal{M}}(G). \]
This equality is quite remarkable since the invariant on the left only involves the torsion-free part of $G$, while the invariant on the right uses the full structure of finite subgroups of $G$ in its definition.\\

There are several important classes of groups for which one has  $\underline{\mathrm{cd}}_{\mathcal{M}}(G)= \underline{\mathrm{cd}}(G)$. For instance, equality holds for elementary amenable groups of type $FP_\infty$ \cite{KMPN}, lattices in classical simple Lie groups \cite{ADMS}, mapping class groups \cite{Armart}, outer automorphism groups of free groups \cite{Luck2,Vogtmann} and for groups that act properly and chamber transitively on a building, such as Coxeter groups and graph products of finite groups \cite{DMP}. However, there are also groups for which  one has a strict inequality $\underline{\mathrm{cd}}_{\mathcal{M}}(G)< \underline{\mathrm{cd}}(G)$. Indeed, such examples have been constructed in  \cite{LearyNucinkis} and \cite{LP} (see also \cite{Martinez12}, \cite{DP2}) and arise as certain semi-direct products of Bestvina-Brady groups or right angled Coxeter groups with finite groups. These examples also show that one can have a strict inequality $\underline{\mathrm{cd}}_{\mathcal{M}}(G)< \underline{\mathrm{cd}}(G)$ for CAT(0)-groups and word-hyperbolic groups, and that the gap between these two invariants can be arbitrary big. Hence, even for groups $G$ that are very well behaved from many perspectives, e.g.~ they have strong cohomological finiteness properties and nice metric properties, the invariants  $\underline{\mathrm{cd}}_{\mathcal{M}}(G)$ and $\underline{\mathrm{cd}}(G)$ can be quite different. The most general statement about their relationship that is known at the moment is 
\begin{equation} \label{eq: length bound}\underline{\mathrm{cd}}_{\mathcal{M}}(G)  \leq \underline{\mathrm{cd}}(G) \leq \max_{H \in \mathcal{F}}\Big\{  \underline{\mathrm{cd}}_{\mathcal{M}}(W_G(H))+l(H)\Big\},  \end{equation}
where the length $l(H)$ of a finite group $H$ is the length of the longest chain of subgroups of $H$ and $W_G(H)=N_G(H)/H$ is the Weyl group of $H$ in $G$. Moreover, this upper bound is attained (see (\cite[Th. A ]{Degrijse})). In particular, if there is uniform bound on the length of finite subgroups of $G$, then the finiteness of $\underline{\mathrm{cd}}_{\mathcal{M}}(G)$ implies the finiteness of $\underline{\mathrm{cd}}(G)$, since $\underline{\mathrm{cd}}_{\mathcal{M}}(W_G(H))\leq \underline{\mathrm{cd}}_{\mathcal{M}}(G)$ for all $H\in \mathcal{F}$ (see \cite[eq. (11) \& Lemma 5.1]{Degrijse}).  However, if there is no such bound for a certain group $G$, then it is still an open problem whether or not one can have $\underline{\mathrm{cd}}_{\mathcal{M}}(G)<\infty$ but $\underline{\mathrm{cd}}(G)=\infty$.  \\

Finally, it is also worth mentioning that one always has
\[ \mathrm{cd}_{\mathbb{Q}}(G)\leq \underline{\mathrm{cd}}_{\mathcal{M}}(G), \]
where $\mathrm{cd}_{\mathbb{Q}}(G)$ denotes the rational cohomological dimension of $G$. This follows from the fact that if $F_{\ast}\rightarrow \underline{A}$ is a free resolution in $\mathrm{Mack}_{\mathcal{F}}G$, then $F(G/e)\otimes_{\mathbb{Z}}\mathbb{Q} \rightarrow A(G/e)\otimes_{\mathbb{Z}}\mathbb{Q}=\mathbb{Q} $ is a projective $\mathbb{Q}[G]$-resolution of $\mathbb{Q}$. Note that this inequality can be strict (e.g. see \cite[Ex. 8.5.8]{DavisBook}) and that it implies that $\underline{\mathrm{cd}}_{\mathcal{M}}(G)=0$ if and only if $G$ is finite.

\section{Proper $G$-spaces and proper $G$-spectra}
Throughout this section, let $G$ be a discrete group and let $\mathcal{F}$ be the family of finite subgroups of $G$. If $X$ is a $G$-space, then $X_{+}$ denotes the space obtained by adding a disjoint $G$-fixed basepoint $\ast$ to $X$. In what follows we will freely make use of the language of triangulated categories (e.g.~see \cite{Gelfand-Manin}) and model categories (e.g.~see \cite{Hovey}).

Denote by $G\mbox{-}\mathrm{Top}^{\mathcal{F}}_{\ast}$ the model category of compactly generated weak Hausdorff spaces equipped with a continuous $G$-action and $G$-fixed basepoint together with $G$-equivariant based continuous maps, where weak-equivalences and fibrations are required to be weak-equivalences and fibrations on $H$-fixed point spaces, for all $H \in \mathcal{F}$ (see \cite[Section 2.1-2.3]{Fau}). The associated homotopy category is denoted by $\mathrm{Ho}(G\mbox{-}\mathrm{Top}^{\mathcal{F}}_{\ast})$. Note that if $H$ is a subgroup of $G$, then $G/H_{+}$ is a cofibrant object in $G\mbox{-}\mathrm{Top}^{\mathcal{F}}_{\ast}$ if and only if $H \in \mathcal{F}$. In particular $G/G_{+}=S^{0}$ is not a cofibrant object in $G\mbox{-}\mathrm{Top}^{\mathcal{F}}_{\ast}$ if $G$ is infinite. 

Now let $X$ be a model for $\underline{E}G$. Since this by definition means that $X$ is a proper $G$-CW-complex such that the map $X \rightarrow G/G$ is a weak equivalence on $H$-fixed point sets, for all $H \in \mathcal{F}$, it follows that 
\begin{equation}\label{eq: proper decomp} S^0 \cong \mathrm{colim}_{n} X^n_{+} \end{equation}
in $\mathrm{Ho}(G\mbox{-}\mathrm{Top}^{\mathcal{F}}_{\ast})$, where $ X^{n}=\{+\}$ if $n<0$ while for $n\geq 0$, one has inclusions of based $G$-spaces $ X^{n-1}_{+}\rightarrow  X^n_{+}$ that fit into homotopy cofiber sequences
\[   X ^{n-1}_{+}\rightarrow X^n_{+}\rightarrow X^n/X^{n-1}=\bigvee_{i \in I_n}\Sigma^n G/H_{i \ +} \rightarrow \Sigma X^{n-1}_{+},   \]
such that $H_i \in \mathcal{F}$ for all $i \in I_n$ and $\Sigma^n$ denotes the smash product with $S^n$. The isomorphism (\ref{eq: proper decomp}) is called a \emph{proper decomposition of $S^0$}. Note in particular that $X_{+}$ is a cofibrant replacement of $S^0$ in $G\mbox{-}\mathrm{Top}^{\mathcal{F}}_{\ast}$ and that the dimension of $X$ as a $G$-CW-complex corresponds to the smallest $d$ such that $X^{n-1}\xrightarrow{\cong} X^{n}$ for all $n\geq d+1$. Note also that $X$ is a $G$-CW-complex of finite type (meaning that the orbit space $G\setminus X$ has finitely many cells in each dimension) if and only if the sets $I_n$ are finite for all $n$. \\

We recall from \cite[Section 6]{Fau} that a \emph{$G$-spectrum} is an orthogonal spectrum $X$ equipped with a $G$-action
\[    G \rightarrow \mathrm{Aut}( X).  \]
A morphism of $G$-spectra is a morphism of the underlying orthogonal spectra that commutes with the $G$-action. The category of $G$-spectra is denoted by $\Gspec$.
For every $H \in \mathcal{F}$, one can consider the category $\Hspec$ of orthogonal $H$-spectra \cite{MandellMay} and the restriction functor
\[   \mathrm{res}_H^G : \Gspec \rightarrow  \Hspec :  X \mapsto  \mathrm{res}_H^G( X).   \]
Any orthogonal $H$-spectrum $Y$ can be evaluated on an arbitrary finite dimensional orthogonal $H$-representation $V$. The $H$-space $Y(V)$ is defined by 
$$Y(V) = L(\mathbb{R}^n, V)_+ \wedge_{O(n)} Y_n,$$
where the number $n$ is the dimension of $V$, the vector space $\mathbb{R}^n$ is equipped with the standard scalar product and $L(\mathbb{R}^n, V)$ is the space of (not necessarily equivariant) linear isometries from $\mathbb{R}^n$ to $V$. The $H$-action on $Y(V)$ is diagonal.
This construction allows one to define genuine $H$-\emph{equivariant homotopy groups} 
$$\pi_k^H Y = \colim_{V \subset \mathcal{U} }[S^{\mathbb{R}^k \oplus V}, Y(V)]^H, \;\; k \in \mathbb{Z},$$
where $V$ ranges over all finite dimensional $H$-subrepresentations of a complete $H$-universe $\mathcal{U}$ (see \cite[Section III.3]{MandellMay} for details). Given a $G$-spectrum $X$ and a finite subgroup $H$ of $G$, one defines $\pi_k^HX$ to be $\pi_k^H \mathrm{res}_H^G( X)$. Finally, a map $f \colon X \rightarrow X'$ of $G$-spectra is called a \emph{(proper) stable equivalence} if the induced map
$$\pi_k^H(f) \colon \pi_k^H X \rightarrow \pi_k^H X'$$
is an isomorphism for any integer $k$ and any $H \in \mathcal{F}$.

It follows from \cite[Section 6]{Fau} that there is a stable model category structure on $\Gspec$ with weak equivalences the stable equivalences. We refer to this model category as the model category of \emph{proper $G$-spectra}. The term \emph{proper} here does not refer to the action of $G$ on $X$, but rather to the form of weak equivalences in $\Gspec$. 

The homotopy category associated to the stable model category $\Gspec$ will be denoted by $\mathrm{Ho}(\Gspec)$. The category $\mathrm{Ho}(\Gspec)$ is naturally a triangulated category (see \cite[Chapter 7]{Hovey}), whose distinguished triangles will be called \emph{stable cofiber sequences} and whose suspension functor will be denoted by $\Sigma$, while its inverse will be denoted by $\Sigma^{-1}$ as usual. The abelian group of morphisms from  $ X$ to $ Y$ in $\mathrm{Ho}(\Gspec)$ will be denoted by $[ X, Y]^G$. Recall that an object $ X$ of $\mathrm{Ho}(\Gspec)$ is called \emph{compact} if the functor $[ X, -]^G$ preserves infinite coproducts.

There is a suspension functor (see \cite[Section 6]{Fau})
\[    \Sigma^{\infty}: G\mbox{-}\mathrm{Top}^{\mathcal{F}}_{\ast} \rightarrow \Gspec: X \mapsto \Sigma^{\infty}X \]
that is a left Quillen functor and hence preserves cofibrations and weak equivalences between cofibrant objects. Thus it yields a derived functor

\[    \Sigma^{\infty}: \mathrm{Ho}(G\mbox{-}\mathrm{Top}^{\mathcal{F}}_{\ast}) \rightarrow  \mathrm{Ho}(\Gspec) \]
that sends homotopy cofiber sequences to stable cofiber sequences. If $H$ is a subgroup of $G$, we will abuse notation slightly and sometimes denote $\Sigma^{\infty}G/H_{+}$ by $G/H_{+}$, hoping that it will be clear from the context what is meant. In particular, letting $H=G$, we will denote $\Sigma^{\infty}G/G_{+}$ by $S^0$. Note that $G/H_{+}$ is a cofibrant object in $\Gspec$ if and only if $H \in \mathcal{F}$. In particular, $S^0$ is not a cofibrant object in $\Gspec $ if $G$ is infinite. Recall that a cofibrant replacement of $S^0$ in $\Gspec $ is a cofibrant object that is weakly equivalent to $S^0$, i.e.~isomorphic to $S^{0}$ in $\mathrm{Ho}(\Gspec)$.

\begin{definition} \label{def: stable model} \rm A \emph{stable model for} $\underline{E}G$ consists of a collection of $G$-spectra $\{X^n\}_{n\in \mathbb{Z}}$ together with a collection of morphisms $ X^{n-1}\rightarrow  X^n$ in $\mathrm{Ho}(\Gspec)$, where $ X^{n}=\{\ast\}$ if $n<0$ while for each $n\geq 0$ there exists a stable cofiber sequence
\[      X^{n-1} \rightarrow  X^{n} \rightarrow   \bigvee_{i \in I_n} \Sigma^n G/H_{i \ +} \rightarrow \Sigma  X^{n-1}   \]
such that $H_i \in \mathcal{F}$ for all $i \in I_n$, and 
\begin{equation}   \label{eq: def stable model}  \mathrm{hocolim}_n  X^n \cong S^0 \end{equation}
in $\mathrm{Ho}(\Gspec)$. In analogy with the unstable case, we call the isomorphism (\ref{eq: def stable model}) a \emph{stable proper decomposition of $S^0$}. If there exists a $d$ such that $X^{n-1}\xrightarrow{\cong} X^{n}$ for all $n\geq d+1$, then the stable model for $\underline{E}G$ is called \emph{finite dimensional}. In this case, the smallest such $d$ is called the \emph{dimension} of the stable model for $\underline{E}G$. A stable model for $\underline{E}G$ is said to be of \emph{finite type} if the sets $I_n$ are finite for all $n\geq 0$. A stable model for $\underline{E}G$ is a said to be \emph{finite} if it is both finite dimensional and of finite type. 
\end{definition}
\begin{remark}\rm \label{remark: terminology}In what follows we shall sometimes abuse terminology and refer to a certain $G$-spectrum $X$ as a stable model for $\underline{E}G$. It should be understood that in this case we mean that $X$ is a specific model for the homotopy colimit $\mathrm{hocolim}_n  X^n$ where the $G$-spectra $X^n$ satisfy all the assumptions of Definition \ref{def: stable model}. In other words, we will often threat a stable model for $\underline{E}G$ as a $G$-spectrum, keeping in mind that the specific $G$-spectra $X^n$ and morphisms $X^{n-1}\rightarrow X^{n}$ are part of the data. In particular, if the collection of $G$-spectra $\{X^n\}_{n\in \mathbb{Z}}$ form an $m$-dimensional stable model for $\underline{E}G$, then we might as well take $X^n=X^m$ for all $n\geq m$  and refer to $X^m=\mathrm{hocolim}_n  X^n$ as an $m$-dimensional stable model for $\underline{E}G$.

\end{remark}

The triangulated category $\mathrm{Ho}(\Gspec)$ is compactly generated by the set of compact generators
\[\{G/H_{+} \; |\; H \in \mathcal{F} \}.\]
This follows from the fact that for $H \in \mathcal{F}$, $n \in \mathbb{Z}$ and $ X \in \mathrm{Ho}(\Gspec)$, there is a natural isomorphism (see \cite[Lemma 6.11]{Fau})
\[[\Sigma^n G/H_{+},  X]^G \cong [S^n,  X]^H =  \pi^{H}_{n}( X).\] 
Using the double coset formula and the Wirthm\"uller isomorphism \cite{MayB}, the latter also implies that for finite subgroups $H$ and $K$ of $G$, we have an isomorphism
\[[\Sigma^{\infty} G/K_{+}, \Sigma^{\infty} G/H_{+}]_*^G \cong  \bigoplus_{[g] \in K \setminus G / H} \pi_*^{K \cap {}^g H} (S^0).\]
In particular when $\ast=0$, we get
\[[\Sigma^{\infty} G/K_{+}, \Sigma^{\infty} G/H_{+}]^G \cong  \bigoplus_{[g] \in K \setminus G / H} A({K \cap {}^g H}),\]
where we remind the reader that $A(K \cap {}^g H)$ denotes the Burnside ring of the finite group $K \cap {}^g H$. It now follows from \cite[Proposition 3.1]{MartinezNucinkis06} that the Mackey category $\mathcal{M}_{\mathcal{F}}G$ fully faithfully embeds into $\mathrm{Ho}(\Gspec)$ by sending $G/H$ to $\Sigma^{\infty}G/H_{+}$, for any $H \in \mathcal{F}$. Summarizing the discussion we see that the abelian groups 
\[ \pi^{H}_{n}( X) \cong [\Sigma^n G/H_{+},  X]^G\]
assemble together to form a Mackey functor
\[     \pi^{-}_{n}( X) : \mathcal{M}_{\mathcal{F}}G \rightarrow \mathbb{Z}\mbox{-mod}: G/H \mapsto  \pi^{H}_{n}( X) \]
such that the functor $\pi^{-}_{0}(G/K_+)$ is isomorphic to the free Mackey functor $\mathbb{Z}^G[-,K]$ and the functor  $\pi^{-}_{0}(S^0)$ is isomorphic to the Burnside ring functor $\underline{A}$. 

Next we note that the observations above together with \cite[Theorem 3.4]{MandellMay} imply that any stable cofiber sequence
\[      X \rightarrow  Y \rightarrow  Z \xrightarrow{\partial} \Sigma X \]
in $\mathrm{Ho}(\Gspec)$ induces an exact sequence of Mackey functors
\[  \pi^{-}_{n}( X) \rightarrow  \pi^{-}_{n}( Y) \rightarrow  \pi^{-}_{n}( Z) \xrightarrow{\partial_{\ast}}  \pi^{-}_{n}(\Sigma X) \rightarrow  \pi^{-}_{n}( \Sigma Y) \]
for every $n \in \mathbb{Z}$. The suspension and desuspension isomorphisms
\[   \pi^{-}_{n}(\Sigma X) \cong  \pi^{-}_{n-1}( X)      \]
and 
\[   \pi^{-}_{n}(\Sigma^{-1} X) \cong  \pi^{-}_{n+1}( X)      \]
can be used to splice these exact sequence together and form long exact sequences.\\

Another consequence of the fact that $\mathrm{Ho}(\Gspec)$ is compactly generated by the set
\[\{G/H_{+} \; |\; H \in \mathcal{F} \}\]
is the existence of Eilenberg-MacLane objects for Mackey functors. Given a Mackey functor $M : \mathcal{M}_{\mathcal{F}}G \rightarrow \mathbb{Z}\mbox{-mod}$, there exists a $G$-spectrum $HM$, called the \emph{Eilenberg-MacLane spectrum of $M$} such that $\pi^{-}_{0}(HM)$ is isomorphic as a Mackey functor to $M$ and $\pi^{-}_{n}(HM)=0$ if $n \neq 0$. Moreover, $HM$ is unique up to stable homotopy. This follows from \cite[Proposition 3.8]{Schwede} which asserts that under certain general conditions, a set of compact generators in any triangulated category with infinite sums determines a $t$-structure (see also \cite{Alonso}). The Eilenberg-MacLane objects are then just the objects of the heart of this $t$-structure. The paper \cite{DHLPS} shows that the spectrum $HM$ represents the $G$-equivariant Bredon cohomology with coefficients in $M$. \\

Finally, in light of Remark \ref{remark: terminology}, we can say that a stable model for $\underline{E}G$ is a proper cellular stable decomposition of $S^0$ and that finite stable models for $\underline{E}G$ are compact objects of $\mathrm{Ho}(\Gspec)$ because they can be built in finitely many steps by iterated stable cofiber sequences from shifts of the suspension spectra $\Sigma^{\infty}G/H_{+}$ for suitable finite subgroups $H$ of $G$. The objects $\Sigma^{\infty}G/H_{+}$ are compact as pointed out above and the class of compact objects in any triangulated category is closed under 2-out-of-3 in distinguished triangles, so the claim follows. By the properties of the functor $\Sigma^{\infty}$ listed above and the fact that $S^{0}$ is cofibrant in $H\mbox{-}\mathrm{Top}^{\mathcal{F}}_{\ast}$ for all finite $H$, it follows that if $X$ is a model for $\underline{E}G$, then $\Sigma^{\infty}X_{+}$ is a stable model for $\underline{E}G$. Moreover, if $\dim(X)=d$ then $\Sigma^{\infty}X_{+}$ is $d$-dimensional and if $X$ is of finite type, then so is $\Sigma^{\infty}X_{+}$.

\section{Geometric versus cohomological dimension}
Throughout this section, let $G$ be a discrete group and let $\mathcal{F}$ be the family of finite subgroups of $G$. Recall that the geometric dimension for proper actions $\underline{\mathrm{gd}}(G)$ of $G$ is by definition the smallest possible dimension that a model for $\underline{E}G$ can have and note that if $X$ is a model for $\underline{E}G$, then the cellular chain complexes \[   \ldots \rightarrow C_n(X^{H})  \rightarrow  C_{n-1}(X^{H}) \rightarrow \ldots \rightarrow C_{0}(X^{H}) \rightarrow \mathbb{Z}\rightarrow 0 \]
of the fixed points subspaces $X^H$, for all $H \in \mathcal{F}$, assemble to form a free resolution $C_{\ast}(X^{-})\rightarrow \underline{\mathbb{Z}}$ in $\orbmod$ (see \cite{LuckMeintrup}). This implies that $\underline{\mathrm{cd}}(G)\leq \underline{\mathrm{gd}}(G)$. In \cite[Th. 0.1.]{LuckMeintrup} it is shown that one even has
\[   \underline{\mathrm{cd}}(G)\leq \underline{\mathrm{gd}}(G) \leq \max\{3,\underline{\mathrm{cd}}(G)\}.    \]
It is not hard to check that $\underline{\mathrm{gd}}(G)=0$ if and only if $\underline{\mathrm{cd}}(G)=0$ if and only if $G$ is finite. Since one has $\underline{\mathrm{cd}}(G)=1$ if and only if $\underline{\mathrm{gd}}(G)=1$ by \cite[Cor. 1.2]{Dunwoody}, we conclude that the invariants  $\underline{\mathrm{cd}}(G)$ and $\underline{\mathrm{gd}}(G)$ coincide, except for the possibility that one could have  $\underline{\mathrm{cd}}(G)=2$ but  $\underline{\mathrm{gd}}(G)=3$. That this Eilenberg-Ganea exception actually occurs in the torsion setting is shown in \cite{BradyLearyNucinkis} (see also Section \ref{sec: Examples}). We conclude that in the unstable world there is a nice geometric interpretation for the algebraic invariant $\underline{\mathrm{cd}}(G)$. The main purpose of this section is the show that there is a similar geometric interpretation for the algebraic invariant  $\underline{\mathrm{cd}}_{\mathcal{M}}(G)$ in the stable world.

\begin{definition} \rm  The \emph{stable geometric dimension for proper actions of $G$}, denoted by  $\underline{\mathrm{gd}}_{\mathrm{st}}(G)$, is the smallest integer $d$ such that there exists a $d$-dimensional stable model for $\underline{E}G$. If such an integer does not exist, then we set $\underline{\mathrm{gd}}_{\mathrm{st}}(G)=\infty$.

\end{definition}

\begin{theorem}\label{th: main theorem} For any discrete group $G$, one has
\[    \underline{\mathrm{cd}}_{\mathcal{M}}(G)  =  \underline{\mathrm{gd}}_{\mathrm{st}}(G).    \]
In particular, if $G$ is virtually torsion-free then
\[   \mathrm{vcd}(G) =  \underline{\mathrm{gd}}_{\mathrm{st}}(G).    \]
\end{theorem}
The next two subsections will be devoted to proving this theorem. The proof will also provide us with a method to construct stable models for $\underline{E}G$ that do not (necessarily) come from suspending unstable models. This will be illustrated in Section 6.\\

\subsection{The proof of $\leq$ } \label{easy}
If $\underline{\mathrm{gd}}_{\mathrm{st}}(G)=\infty$, then there is nothing to prove. So let's assume that $\underline{\mathrm{gd}}_{\mathrm{st}}(G)=m$ is finite and let $ X^m$ be an $m$-dimensional stable model for $\underline{E}G$ (see Remark \ref{remark: terminology}). For each $n\geq 1$, consider the stable cofiber sequences
\[    X^{n-2} \xrightarrow{i}  X^{n-1} \xrightarrow{\pi}    X^{n-1}/ X^{n-2} \xrightarrow{\partial} \Sigma   X^{n-2}  \]
and 
\[    X^{n-1} \xrightarrow{i}  X^{n} \xrightarrow{\pi}    X^{n}/ X^{n-1}  \xrightarrow{\partial} \Sigma   X^{n-1}.  \]
and define
\[      d_n:    \pi^{-}_{n}( X^n/ X^{n-1}) \xrightarrow{\partial_{\ast}}   \pi^{-}_{n}(\Sigma X^{n-1}) \xrightarrow{-(\Sigma \pi){\ast}} \pi^{-}_{n}(\Sigma X^{n-1}/ X^{n-2}) \cong  \pi^{-}_{n-1}( X^{n-1}/ X^{n-2}).   \]
Since $\partial \circ \pi = 0$, one has $d_n \circ d_{n-1}=0$ and
\[     \pi^{-}_{n}( X^n/ X^{n-1}) \cong \pi^{-}_{n}(\bigvee_{i\in I_n} \Sigma^n G/H_{i \ +}) \cong \bigoplus_{i\in I_n} \pi^{-}_{n}(\Sigma^n G/H_{i \ +}) \cong\bigoplus_{i\in I_n} \mathbb{Z}^{G}[-,H_i],    \]
Hence we obtain a chain complex of free Mackey functors
\begin{equation}\label{eq: res} 0 \rightarrow  \pi^{-}_{m}( X^m/ X^{m-1}) \xrightarrow{d_m}  \pi^{-}_{m-1}( X^{m-1}/ X^{m-2})  \xrightarrow{d_{m-1}} \ldots   \xrightarrow{d_2}\pi^{-}_{1}( X^{1}/ X^{0})  \xrightarrow{d_1}  \pi^{-}_{0}( X^{0})  \xrightarrow{d_0} 0.   \end{equation}
We claim that the homology of this chain complex is zero, except in degree zero where it is isomorphic to the Burnside ring functor $\underline{A}$ via the the inclusion $X^0 \rightarrow X^m\cong S^0$.

To prove this claim, let us first fix $K \in \mathcal{F}$ and denote $A=\mathrm{res}_K^G\underline{A}$ . Recall from \cite{LewisMayMcClure} that Bredon homology of $K$-spectra with Burnside ring coefficients, denoted by $\mathrm{H}^{K}_{\ast}(-)$, is represented by an Eilenberg-Maclane $K$-spectrum $ HA$. By \cite[Th. 2.1]{Lewis1}, the unit map $S^{0}\rightarrow  HA$ induces a stable Hurewicz isomorphism
\[      \pi^{K}_{0}( D) \xrightarrow{\cong}  \pi^{K}_{0}( D\wedge  HA)=\mathrm{H}^K_{0}( D)   \]
for every connective $K$-spectrum $ D$. Denote  $Y^n=\mathrm{res}^G_K( X^n) $ for all $n\in \mathbb{Z}$. Since $\mathrm{res}^G_K$ is exact and
\[     \mathrm{res}^G_K(G/H_{i \ +})=\bigvee_{[g] \in K \setminus G/ H_i} K/K\cap {}^gH_{i \ +}\]
it follows that $ Y^m=\mathrm{hocolim}_n  Y^n$ is a $K$-CW-spectrum that is isomorphic to $S^0=K/K_{+}$ in $\mathrm{Ho}(\Kspec)$ and whose cellular decomposition is given by restricting the stable cofiber  sequences of $ X$ to $\mathrm{Ho}(\Kspec)$. For each $n$, there is a commutative diagram of cofiber sequences of $K$-spectra
\[\xymatrix{ Y^{n-1} \ar[r] \ar[d] &  Y^n \ar[r] \ar[d] &  Y^n/ Y^{n-1} \ar[r] \ar[d] & \Sigma  Y^{n-1} \ar[d]\\
 Y^{n-1}\wedge  HA \ar[r]  &  Y^n\wedge  HA  \ar[r] &  Y^n/ Y^{n-1}\wedge  HA  \ar[r]  & \Sigma  Y^{n-1}\wedge  HA.  }\]
Applying $K$-homotopy groups to this sequence and using the fact that $\mathrm{res}^G_K(G/H_{i \ +})$ is a connective $K$-spectrum, we conclude that the chain complex (\ref{eq: res}) evaluated at $K$ is naturally isomorphic to
\begin{equation*} 0 \rightarrow  \mathrm{H}^{K}_{m}( Y^m/ Y^{m-1}) \xrightarrow{d_m}  \mathrm{H}^{K}_{m-1}( Y^{m-1}/ Y^{m-2})  \xrightarrow{d_{m-1}} \ldots   \xrightarrow{d_2}\mathrm{H}^{K}_{1}( Y^{1}/ Y^{0})  \xrightarrow{d_1}  \mathrm{H}^{K}_{0}( Y^{0})  \xrightarrow{d_0} 0.   \end{equation*}
Since this is the cellular chain complex of $ Y^m$, its homology computes \[\mathrm{H}^K_{\ast}( Y^m)\cong \mathrm{H}^K_{\ast}(S^{0})\cong \pi^K_{\ast}( HA)\] (e.g. see \cite[Ch. XIII]{May}). This shows that (\ref{eq: res}), evaluated at $K \in \mathcal{F}$, has zero homology except in degree zero, where it is isomorphic to $A(K)$ via the inclusion $Y^0 \rightarrow Y^m\cong S^0$. This proves our claim. 

We conclude that there is a free resolution 
\begin{equation} \label{eq: tralala} 0 \rightarrow  \pi^{-}_{m}( X^m/ X^{m-1}) \xrightarrow{d_m} \ldots   \xrightarrow{d_2}\pi^{-}_{1}( X^{1}/ X^{0})  \xrightarrow{d_1}  \pi^{-}_{0}( X^{0})  \xrightarrow{d_0} \underline{A}\rightarrow 0   \end{equation}
of $\underline{A}$ in $\mathrm{Mack}_{\mathcal{F}}G$ of lenght $m$, which implies that $\underline{\mathrm{cd}}_{\mathcal{M}}(G)\leq m$ and hence 
\[    \underline{\mathrm{cd}}_{\mathcal{M}}(G)  \leq  \underline{\mathrm{gd}}_{\mathrm{st}}(G).    \]
\subsection{The proof of $\geq$ } \label{sec: hard}
If $\underline{\mathrm{cd}}_{\mathcal{M}}(G) =\infty$ then we are done since $ \underline{\mathrm{cd}}_{\mathcal{M}}(G)  \leq  \underline{\mathrm{gd}}_{\mathrm{st}}(G)  $. If $\underline{\mathrm{cd}}_{\mathcal{M}}(G) = 0$, then $G$ is finite which implies that $\underline{\mathrm{gd}}_{\mathrm{st}}(G) =0$ since in this case $G/G_{+}$ is a stable model for $\underline{E}G$ of dimension zero. Now assume that $\underline{\mathrm{cd}}_{\mathcal{M}}(G) =m \geq 1$ is finite. Let $\underline{E}G$ be a (possible infinite dimensional) classifying space for proper actions. Such a space always exists by \cite[Th. 1.9]{Luck2}. Applying the functor $\Sigma^{\infty}(-)_{+}$ to the weak equivalence $\underline{E}G\rightarrow G/G$ we obtain 
a stable proper decomposition

\[  f: X=\mathrm{hocolim}_n X^n \xrightarrow{\cong} S^0 \] 
in $\mathrm{Ho}(\Gspec)$. Just as in the previous section, $ X$ gives rise to a free resolution (possibly of infinite length) 
\[   \ldots \rightarrow  \pi^{-}_{n}( X^n/ X^{n-1}) \xrightarrow{d_n} \ldots   \xrightarrow{d_2}\pi^{-}_{1}( X^{1}/ X^{0})  \xrightarrow{d_1}  \pi^{-}_{0}( X^{0})  \xrightarrow{d_0} \underline{A}\rightarrow 0       \]
of $\underline{A}$ in $\mathrm{Mack}_{\mathcal{F}}G$.  Since $\underline{\mathrm{cd}}_{\mathcal{M}}(G)=m$, it follows from standard techniques in homological algebra that $\ker d_{m-1}$ is a projective Mackey functor. By the Eilenberg swindle there exists a free Mackey functor 
\[ F'=\bigoplus_{j \in J} \mathbb{Z}^{G}[-,H_j] \] such that $F=F'\oplus \ker d_{m-1}$ is a free Mackey functor. Hence, by replacing $ X^{m-1}$ with \[ X^{m-1}\vee \big(  \bigvee_{j \in J} \Sigma^{m-1}G/H_{j \ +} \big)\] 
and letting $f( \bigvee_{j \in J} \Sigma^{m-1}G/H_{j \ +} )=\ast$,
we may assume that $\ker{d_{m-1}}$ is a free Mackey functor $F$ so that we obtain a free resolution of length $m$
\begin{equation}  \label{eq: freeres1} 0\rightarrow F \rightarrow  \pi^{-}_{m-1}( X^{m-1}/ X^{m-2}) \xrightarrow{d_{m-1}} \ldots   \xrightarrow{d_2}\pi^{-}_{1}( X^{1}/ X^{0})  \xrightarrow{d_1}  \pi^{-}_{0}( X^{0})  \xrightarrow{d_0} \underline{A}\rightarrow 0       \end{equation}
of $\underline{A}$ in $\mathrm{Mack}_{\mathcal{F}}G$ and a map
\[   f:  X^{m-1} \rightarrow S^{0}  \]
that is $(m-1)$-connected. Here, $(m-1)$-connected means that for all $K \in \mathcal{F}$, $\pi^{K}_{i}(f)$ is an isomorphism for all $i\leq m-2$, while $\pi^{K}_{m-1}(f)$ is surjective. For every $-1\leq n\leq m-1$, the $G$-spectrum $ X^{n}$ fits into a stable cofiber sequence
\[      X^{n} \xrightarrow{f} S^{0} \rightarrow  Y^{n+1} \rightarrow \Sigma  X^{n}.   \]
Using the octahedral axiom for triangulated categories, we deduce the existence of the dotted arrows in the commutative diagram
\begin{equation} \label{octa}   \xymatrix{ X^{n-1} \ar[r] \ar@{=}[d]  &  X^n \ar[r]^{\pi \ \ \ \ \ \ \ \ \ \  \ \ } \ar[d]^{f} & \bigvee_{\alpha \in I_n} \Sigma^{n} G/H_{\alpha \ +} \ar[r] \ar@{.>}[d] & \Sigma  X^{n-1} \ar@{=}[d] \\
 X^{n-1} \ar[r]^{f} & S^{0} \ar[r]  \ar[d] &  Y^n \ar@{.>}[d]^{\alpha_n} \ar[r]^{\partial} & \Sigma X^{n-1} \ar[d] \\
&  Y^{n+1} \ar[d]^{\partial} \ar@{=}[r] &  Y^{n+1} \ar[r]^{\partial}  \ar[d]^{\Sigma \pi \circ \partial}& \Sigma  X^{n} \\
& \Sigma X^{n} \ar[r]^{\Sigma \pi \ \ \ \ \ \ \ \ \ \ \ \ } & \bigvee_{\alpha\in I_n} \Sigma^{n+1} G/H_{\alpha \ +}, &  }\end{equation}
where the two upper horizatonal lines and the two middle vertical lines are stable cofiber sequences. Note that $ Y^0=S^{0}$ and that by rotating, we obtain for each $0\leq n\leq m-1$, a stable cofiber sequence
\[     Y^{n} \xrightarrow{\alpha_n}  Y^{n+1} \xrightarrow{\Sigma\pi\circ\partial}\bigvee_{\alpha\in I_n} \Sigma^{n+1} G/H_{\alpha \ +} \rightarrow \Sigma  Y^n. \]
As before, these cofiber sequences give rise to a chain complex 
\begin{equation*}  0 \rightarrow   \pi^{-}_{m}( Y^{m}/ Y^{m-1}) \rightarrow \ldots   \rightarrow\pi^{-}_{1}( Y^{1}/ Y^{0})  \rightarrow  \underline{A}=\pi^{-}_{0}( Y^{0})  \rightarrow 0       \end{equation*}
in $\mathrm{Mack}_{\mathcal{F}}G$ where the  $\pi^{-}_{n}( Y^{n}/ Y^{n-1}) $ are free Mackey functors for $n\geq 1$. Using the same stable Hurewicz argument as in Subsection \ref{easy}, one deduces that this chain complex is naturally isomorphic to 
\[   0\rightarrow   \mathrm{H}^{-}_{m}( Y^{m}/ Y^{m-1}) \rightarrow \ldots   \rightarrow \mathrm{H}^{-}_{1}( Y^{1}/ Y^{0})  \rightarrow  \underline{A}  \rightarrow 0.       \]
Moreover, for each $K \in \mathcal{F}$ this chain complex evaluated at $K$ computes $\mathrm{H}^{K}_{\ast}( Y^m)$. Since the map $f:  X^{m-1} \rightarrow S^{0}$ is $(m-1)$-connected, standard long exact sequence arguments imply that $\pi^{-}_{n}( Y^m)=0$ for all $n\leq m-1$. An application of the stable Hurewicz theorem (see \cite[Th. 2.1]{Lewis1}) then yields that for each $K \in \mathcal{F}$, $\mathrm{H}^{K}_{n}( Y^m)=0$ for all $n\leq m-1$ and the Hurewicz map
\[  \pi^{K}_{m}( Y^m) \rightarrow \mathrm{H}^{K}_{m}( Y^m)     \]
is an isomorphism. Moreover, the commutativity of (\ref{octa}) implies that there is an isomorphism of chain complexes
\begin{equation} \label{eq: isochain} \xymatrix{0 \ar[r] & \pi^{-}_{m}( Y^m)  \ar[r] & \pi^{-}_{m}( Y^{m}/ Y^{m-1}) \ar[r]  & \ldots \ar[r] &    \pi^{-}_{1}( Y^{1}/ Y^{0})  \ar[r] & \underline{A} \ar[r] &   0   \\
0 \ar[r] &  F \ar[u]^{\cong} \ar[r] &  \pi^{-}_{m-1}( X^{m-1}/ X^{m-2}) \ar[r] \ar[u]^{\cong} & \ldots \ar[r]  &    \pi^{-}_{0}( X^{0})  \ar[r] \ar[u]^{\cong} & \underline{A} \ar[r] \ar[u]^{\mathrm{Id} }&   0.   }\end{equation}
Since $F$ is a free Mackey functor, we have
\[  F \cong \bigoplus_{\alpha \in I_m} \mathbb{Z}^{G}[-,H_{\alpha}],   \]
so the isomorphism $F \cong \pi^{-}_{m}( Y^m) $ yields a map
\begin{equation} \label{eq: attach}    \bigvee_{\alpha \in I_m} \Sigma^m G/H_{\alpha \ +} \rightarrow  Y^m  \end{equation}
that induces an isomorphism on $\pi_m^{-}$. Now consider the composite
\[     \bigvee_{\alpha \in I_m} \Sigma^m G/H_{\alpha \ +} \rightarrow  Y^m \xrightarrow{\partial} \Sigma X^{m-1}\]
and desuspend it, by applying $\Sigma^{-1}$, to obtain the map
\[      \bigvee_{\alpha \in I_m} \Sigma^{m-1} G/H_{\alpha \ +}   \xrightarrow{\omega}  X^{m-1} . \] 
We now define $ X^m$ to be the $G$-spectrum that fits into a stable cofiber sequence
  \[      \bigvee_{\alpha \in I_m} \Sigma^{m-1} G/H_{\alpha \ +}   \xrightarrow{\omega}  X^{m-1} \rightarrow  X^m \rightarrow   \bigvee_{\alpha \in I_m} \Sigma^{m} G/H_{\alpha \ +} .      \]
This determines $X^m$ uniquely up to non canonical isomorphism. By rotating, we also obtain a stable cofiber sequence
\[       X^{m-1} \rightarrow  X^m \rightarrow   \bigvee_{\alpha \in I_m} \Sigma^{m} G/H_{\alpha \ +}  \rightarrow \Sigma  X^{m-1}.      \]

Now define $Y^{m+1}$ as the mapping cone of the map (\ref{eq: attach}) and use the octahedral axiom for triangulated categories to deduce the existence of the dotted arrows in the diagram
\begin{equation*} \label{octa2}   \xymatrix{ \bigvee_{\alpha \in I_m} \Sigma^{m-1} G/H_{\alpha \ +} \ar[r] \ar@{=}[d]  & \Sigma^{-1} Y^m \ar[r] \ar[d] & \Sigma^{-1}  Y^{m+1}\ar[r] \ar@{.>}[d] & \bigvee_{\alpha \in I_m} \Sigma^{m} G/H_{\alpha \ +} \ar@{=}[d] \\
 \bigvee_{\alpha \in I_m} \Sigma^{m-1} G/H_{\alpha \ +} \ar[r]^{\ \ \ \ \ \ \ \ \ \ \omega} &  X^{m-1} \ar[r]  \ar[d]^{f} &  X^m \ar@{.>}[d]^{\tilde{f}} \ar[r] & \bigvee_{\alpha \in I_m} \Sigma^{m} G/H_{\alpha \ +} \ar[d]^{\partial} \\
&S^{0}\ar[d] \ar@{=}[r] & S^{0}\ar[r] \ar[d] &  Y^m \\
& Y^m \ar[r] &  Y^{m+1},&  }\end{equation*}
where the two upper horizatonal lines and the two middle vertical lines are stable cofiber sequences. So we obtain an $m$-dimensional proper $G$-CW spectrum $ X^m$ together with a map
\[   \tilde{f}:  X^m \rightarrow S^{0}  \]
that extends $f$. We claim that $\pi^{K}_{\ast}(f)$ is an isomorphism for every $K \in \mathcal{F}$. Proving this is equivalent to showing that $\pi^{K}_{\ast}( Y^{m+1})=0$ for all $K \in \mathcal{F}$. Since $ Y^{m+1}$ is connective, it follows from the stable Hurewicz isomorphism \cite[Th. 2.1]{Lewis1} that this is equivalent to proving that $\mathrm{H}^{K}_{\ast}( Y^{m+1})=0$ for all $K \in \mathcal{F}$. If we evaluate at $K \in \mathcal{F}$, then this homology is computed by the chain complex
\begin{equation}\label{eq: almostdone} 0\rightarrow \mathrm{H}^{K}_{m+1}( Y^{m+1}/ Y^{m}) \xrightarrow{d}   \mathrm{H}^{K}_{m}( Y^{m}/ Y^{m-1}) \rightarrow \ldots   \rightarrow \mathrm{H}^{K}_{1}( Y^{1}/ Y^{0})  \rightarrow   \mathrm{H}^{K}_{0}( Y^{0})    \rightarrow 0.\end{equation}
We already know this sequence is exact up to degree $m-1$. The map $d$ is constructed as the composition
\[   d:  \mathrm{H}^{K}_{m+1}( Y^{m+1}/ Y^{m}) \xrightarrow{\partial_{\ast}}  \mathrm{H}^{K}_{m}( Y^{m})  \rightarrow  \mathrm{H}^{K}_{m}( Y^{m}/ Y^{m-1})  \]
By the stable Hurewicz isomorphism, (\ref{eq: attach}) and the exactness of (\ref{eq: isochain}), we conclude that (\ref{eq: almostdone}) is exact. This proves the claim, so we may conclude that $ X^m$ is an $m$-dimensional stable model for $\underline{E}G$, hence
\[      \underline{\mathrm{cd}}_{\mathcal{M}}(G)  \geq  \underline{\mathrm{gd}}_{\mathrm{st}}(G).   \]
\qed

\subsection{Suspension models} \label{sec: suspension}
We have already seen that if a proper $G$-CW-complex $X$ is a model for $\underline{E}G$, then its suspension spectrum $\Sigma^{\infty}X_{+}$ is a stable model for $\underline{E}G$. The following proposition shows that the converse of this statement also holds if one additionally assumes that $X$ is \emph{$\mathcal{F}$-simply connected}, meaning that $X^K$ is simply connected for every $K \in \mathcal{F}$.  This additional assumption is a necessary one. Indeed, in Section 6 we will give an example of a group $G$ and proper $G$-CW-complex $X$ such that $X$ is not $\mathcal{F}$-simply connected and $\Sigma^{\infty}X_+$ is a stable model for $\underline{E}G$.

\begin{proposition}\label{prop: susp model} Let $X$ be an $\mathcal{F}$-simply connected proper $G$-CW complex. Then $\Sigma^{\infty}X_{+}$ is a stable model for $\underline{E}G$ if and only if $X$ is a model for $\underline{E}G$. 
\end{proposition}

\begin{proof} Suppose that we are given an isomorphism $\Sigma^{\infty}X_{+}\cong S^0$ in $\mathrm{Ho}(\Gspec)$. This implies that for any $H \in \mathcal{F}$, the geometric fixed point spectrum (see \cite[Section V.4]{MandellMay})
\[ \Phi^H (\Sigma^{\infty}X_{+})=\Phi^H (\mathrm{res}^G_H(\Sigma^{\infty}X_{+})) \]
is stably equivalent to $S^0$. By \cite[Corollary 4.6]{MandellMay}, one has
\[\Phi^H (\Sigma^{\infty}X_{+}) \cong \Sigma^{\infty}X^H_{+} \]
which shows that $\Sigma^{\infty}X^H_{+}$ is stably equivalent to the sphere spectrum $S^0$. Since $X^H$ is simply connected, the Whitehead and Hurewicz theorems imply that $X^H$ is contractible. 

\end{proof}

Note that for every based $G$-space $Y$ and every $n\geq 0$, there is a natural transformation
\[   \pi^{-}_{n+2}(\Sigma^2Y) \rightarrow  \mathrm{res}_{\pi}(\pi^{-}_n(\Sigma^{\infty}Y)) \]
of right $\orb$-modules. At $K \in \mathcal{F}$, this map is given by mapping the homotopy class of a $K$-map $f: S^{n+2} \rightarrow S^2 \wedge Y$ to its image in the colimit $\colim_{V \subset \mathcal{U}} [S^{V+n}, S^V\wedge Y]^K$
over a complete universe $\mathcal{U}$ of representations of $K$. Let $X$ be a proper $G$-CW-complex with associated homotopy cofiber sequences 
\[   X^{n-1}_{+} \rightarrow X^{n}_{+} \rightarrow  X^n/X^{n-1} \rightarrow \Sigma X^{n-1}_+. \]
By suspending these homotopy cofiber sequences twice and noting that the suspension gives an isomorphism
\[    \pi^{-}_{n+1}(\Sigma^2 (X^{n-1}/X^{n-2})) \cong   \pi^{-}_{n+2}(\Sigma^3 (X^{n-1}/X^{n-2}))  \]
for every $n\geq 1$, we obtain the chain complex of right $\orb$-modules
\[ \rightarrow  \pi^{-}_{n+2}(\Sigma^2 (X^n/X^{n-1})) \rightarrow  \pi^{-}_{n+1}(\Sigma^2 (X^{n-1}/X^{n-2})) \rightarrow \ldots  \rightarrow  \pi^{-}_{2}(\Sigma^2 X^0_+) \rightarrow 0.  \]
By the Hurewicz theorem, this chain complex is isomorphic to the cellular $\orb$-chain complex of $X$ (see beginning of Section $4$)

\[   \rightarrow C_n(X^{-})  \xrightarrow{d_n}  C_{n-1}(X^{-}) \rightarrow \ldots \rightarrow C_{0}(X^{-}) \rightarrow  0. \]
Now the natural transformation discussed above entails a commutative diagram
\[\xymatrix{ C_n(X^-) \ar[d]\ar[r]^{d_n} &	C_{n-1}(X^-)\ar[d] \\
\mathrm{res}_{\pi}(\pi_n^{-}( \Sigma^{\infty}(X^n/ X^{n-1}))) \ar[r]^{d_n \ \ \ } &	\mathrm{res}_{\pi}(\pi_{n-1}^{-}( \Sigma^{\infty}(X^{n-1}/ X^{n-2}))). }\] 
Using the adjointness of $\mathrm{res}_{\pi}$ and $\mathrm{ind}_{\pi}$ and  $\mathrm{ind}_{\pi}(\mathbb{Z}[-,G/K])=\mathbb{Z}^{G}[-,K]$, we conclude that the chain complex obtained from $ \Sigma^{\infty}X_+$ by applying the methods from Subsection 4.1 to the stable cofiber sequences obtained by applying $\Sigma^{\infty}$ to the homotopy cofiber sequences of $X$, coincides with the chain complex obtained by applying the induction functor $\mathrm{ind}_{\pi}$ to the cellular chain complex $C_{\ast}(X^-)$ of $X$.

The above discussion shows that the suspension spectrum functor is a geometric analog of the induction functor $\mathrm{ind}_{\pi}$. This indicates that there should be an algebraic version of Proposition \ref{prop: susp model}. Indeed, \cite[Th. 3.8]{MartinezNucinkis06} shows that if $P_{\ast}$ is a projective resolution of $\underline{\mathbb{Z}}$, then $\mathrm{ind}_{\pi}(P_{\ast})$ is a projective resolution of $\underline{A}$. Below we show that the converse of the latter is also true. The proof requires the following lemma and uses notation and isomorphisms from Section 2.

\begin{lemma}\label{lemma: tensor zero} For any right $\orb$-module $M$ and any $K \in \mathcal{F}$,  
$\mathrm{ind}_{\pi}(M)(G/K)=0$ implies that $M(G/K)=0$.

\end{lemma}
\begin{proof} Fix $K \in \mathcal{F}$ and let $\mathrm{ind}_K^G$ denote the induction functor from the category of covariant Mackey functors for $K$ to the category of covariant Mackey functors for $G$, associated to the inclusion $i_K^G$ of $K$ into $G$.
Then $$\mathrm{ind}_K^G(\mathbb{Z}^K[K,\pi_K(-)])\cong \mathbb{Z}^{G}[K,\pi_G(-)]$$ where $\pi=\pi_G$ and $\pi_K  $ are the functors defined in (\ref{eq: pi functor}). This can be seen by writing out the definition of the induction functor, noting that
\[      \mathrm{res}_K^G \Big ( \Z[-,G/L]\Big)\cong \bigoplus_{[g]\in K \setminus G /L} \Z[-,K/ K\cap {}^{g}L]    \]
and

\[      \mathrm{res}_K^G\Big ( \Z^G[-,L]\Big)\cong \bigoplus_{[g]\in K \setminus G /L} \Z^K[-,K\cap {}^{g}L]    \]
(see \cite[Proposition 3.1]{MartinezNucinkis06}) for any finite subgroup $L$ of $G$ and using (\ref{eq: tensor with free}). Using this isomorphism, we conclude that

\[ 	\mathrm{ind}_{\pi}(M)(G/K)=M(-)\otimes_{\orb} \mathbb{Z}^{G}[K,\pi_G(-)] \cong M(i_K^G(-))\otimes_{\mathcal{O}_{\mathcal{F}}K} \mathbb{Z}^{K}[K,\pi_K(-)].    \] 
Since $\mathbb{Z}[K/K,-]$ is easily seen to be a direct summand of $\mathbb{Z}^{K}[K/K,\pi(-)]$ as left $\mathcal{O}_{\mathcal{F}}K$-modules, it follows that 	
$\mathrm{ind}_{\pi}(M)(G/K)=0$ implies that \[ M(G/K)=M(i_K^G(-))\otimes_{\mathcal{O}_{\mathcal{F}}K} \mathbb{Z}[K/K,-]=0.\] 
\end{proof}

The following proposition provides the converse of \cite[Th. 3.8]{MartinezNucinkis06}.
\begin{proposition} \label{lemma: hom zero}Let $P_{\ast}$ be a non-negative chain complex of projective right $\orb$-modules such that $\mathrm{H}_{0}(P_{\ast})=\underline{\mathbb{Z}}$. Then $P_{\ast}$ is a projective resolution of $\underline{\mathbb{Z}}$ if  $\mathrm{ind}_{\pi}(P_{\ast})$ is a projective resolution of $\underline{A}$.

\end{proposition}	
\begin{proof} First, let $Q_{\ast}$ be a resolution of $M$ consisting of projective left $\orb$-modules. Then one can consider the double complex $C=P_{\ast}\otimes_{\orb}Q_{\ast} $ whose column and row filtrations give rise to two convergent spectral sequences
	\[   E^2_{p,q}=\mathrm{H}^{h}_p\mathrm{H}^v_q(C) \Rightarrow \mathrm{H}_{p+q}(\mathrm{Tot}(C))  \]
	and 
	\[   E^2_{p,q}=\mathrm{H}^{v}_p\mathrm{H}^h_q(C) \Rightarrow \mathrm{H}_{p+q}(\mathrm{Tot}(C)).  \]
Since $P_p$ is projective, $H^{v}_q(P_p\otimes_{\orb}Q_{\ast})=0$ if $q>0$ and $H^{v}_0(P_p\otimes_{\orb}Q_{\ast})=P_{p}\otimes_{\orb}M$. We therefore conclude from the first spectral sequence that $$\mathrm{H}_{p+q}(\mathrm{Tot}(C))=\mathrm{H}_{p+q}(P_{\ast}\otimes_{\orb} M).$$ Since  $Q_p$ is projective, we have $\mathrm{H}^h_q(P_{\ast}\otimes_{\orb}Q_p)=\mathrm{H}_q(P_{\ast})\otimes_{\orb} Q_{p}$. This implies that the $E_2$-term of the second spectral sequence equals $\mathrm{Tor}^{\orb}_p(\mathrm{H}_q(P_{\ast}),M)$. 	We conclude that there is convergent spectral sequence
\begin{equation} \label{eq: tor-spec} E^2_{p,q}=\mathrm{Tor}^{\orb }_p(\mathrm{H}_q(P_{\ast}),M) \Rightarrow \mathrm{H}_{p+q}(P_{\ast}\otimes_{\orb} M).    \end{equation}
Recall that $P_{\ast}$ is a positive chain complex of projective right $\orb$-modules such that $\mathrm{H}_{0}(P_{\ast})=\underline{\mathbb{Z}}$ and such that $\mathrm{ind}_{\pi}(P_{\ast})$ is a projective resolution of $\underline{A}$. We need to prove that $H_n(P_{\ast})(G/K)=0$ for all $n\geq 1$ and all $K \in \mathcal{F}$. To this end, fix $K \in \mathcal{F}$ and let $M$ be the covariant functor $\mathbb{Z}^G[K,-]$ from $\mathcal{M}_{\mathcal{F}}G$ to abelian groups, restricted to the orbit category via $\pi: \orb \rightarrow \mathcal{M}_{\mathcal{F}}G$. Note that $P_{\ast}\otimes_{\orb} M=\mathrm{ind}_\pi(P_{\ast})(G/K)$ and hence $\mathrm{H}_{n}(P_{\ast}\otimes_{\orb} M)=0$ for all $n>0$. Since $\mathrm{ind}_{\pi}$ takes projective resolutions of $\underline{\mathbb{Z}}$ to projective resolutions of $\underline{A}$ (see \cite[Th. 3.8]{MartinezNucinkis06}) and $\mathrm{H}_{0}(P_{\ast})=\underline{\mathbb{Z}}$, it follows from (\ref{eq: tensor adj}) that 
\[   \mathrm{Tor}^{\orb}_{n}(H_0(P_{\ast}),\mathbb{Z}^{G}[K,\pi(-)])\cong  \mathrm{Tor}^{\mathcal{M}_{\mathcal{F}}G}_{n}(\underline{A},\mathbb{Z}^{G}[K,-])=0   \]
for all $n>0$. The spectral sequence (\ref{eq: tor-spec}) now implies that 
\[E^2_{0,1}=\mathrm{ind}_{\pi}\Big(\mathrm{H}_1(P_{\ast})\Big)(G/K)=0.\]
By Lemma \ref{lemma: tensor zero}, we have that $\mathrm{H}_1(P_{\ast})=0$ showing the the entire $q=1$ row  of (\ref{eq: tor-spec}) is zero. This in turn implies that  \[E^2_{0,2}=\mathrm{ind}_{\pi}\Big(\mathrm{H}_2(P_{\ast})\Big)(G/K)=0.\] Continuing inductively, we can deduce that $H_n(P_{\ast})(G/K)=0$ for all $n\geq 1$, as desired.
\end{proof}

Note that by applying Proposition \ref{lemma: hom zero} to the cellular chain complex of $X$, we can give an alternative proof of Proposition \ref{prop: susp model}.

\section{Finite and finite type models}
In this section we relate the finiteness properties of stable models for $\underline{E}G$ to cohomological finiteness properties in the categories $\orbmod$ and $\mathrm{Mack}_{\mathcal{F}}G$, to the compactness of $S^0$ in $\mathrm{Ho}(\Gspec)$ and to finiteness properties of unstable models for $\underline{E}G$.\\

We first consider stable models of finite type.

\begin{theorem} \label{th: finite type} Let $G$ be a discrete group and let $\mathcal{F}$ be its family of finite subgroups. The following are equivalent.
\begin{itemize}
\item[(1)]  There exists a stable model for $\underline{E}G$ of finite type.

\item[(2)] There exists a resolution of $\underline{A}$ in $\mathrm{Mack}_{\mathcal{F}}G$ consisting of finitely generated free modules.
\item[(3)] There exists a resolution of $\underline{\mathbb{Z}}$ in $\orbmod$ consisting of finitely generated free modules.
\item[(4)] There are only finitely many conjugacy classes of finite subgroups in $G$ and the Weyl group $W_G(H)$ of any finite subgroup $H$ of $G$ admits a $\mathbb{Z}[W_G(H)]$-resolution of $\mathbb{Z}$ consisting of finitely generated free modules, i.e.~$W_G(H)$ is of type $FP_{\infty}$.
\end{itemize}
Morever, if one adds to (1), (2),(3), (4) the assumption that for every $H \in \mathcal{F}$, the Weyl-group $W_G(H)=N_G(H)/H$ is finitely presented, then the resulting statements are equivalent to the existence of an unstable model for $\underline{E}G$ of finite type.

\end{theorem}
\begin{proof}  It follows from \cite[Cor. 3.7]{StGreen} that (2) and (3) are equivalent and from  \cite[Lemma 3.1]{KMPN}  that (3) and (4) are equivalent. If $ X$ is a  stable model for $\underline{E}G$ of finite type then proceeding as in Section \ref{easy} yields a resolution
\[ \ldots \rightarrow  \pi^{-}_{m}( X^m/ X^{m-1}) \xrightarrow{d_m} \ldots   \xrightarrow{d_2}\pi^{-}_{1}( X^{1}/ X^{0})  \xrightarrow{d_1}  \pi^{-}_{0}( X^{0})  \xrightarrow{d_0} \underline{A}\rightarrow 0  \]
of $\underline{A}$ in $\mathrm{Mack}_{\mathcal{F}}(G)$ consisting of finitely generated free modules. This proves that (1) implies (2). Now assume that there exists a free resolution of $\underline{A}$ in $\mathrm{Mack}_{\mathcal{F}}G$ consisting of finitely generated free modules. By \cite[Lemma 3.2]{StGreen} we know that $G$ has only finitely many conjugacy classes of finite subgroups. This implies that there exists a model for $\underline{E}G$ whose zero skeleton consists of finitely many orbits. Via the functor $\Sigma^{\infty}(-)_{+}$, we obtain a stable map
\[ f_0:   X^{0}=\bigvee_{i\in I_{0}}G/H_{i \ +} \rightarrow S^{0}   \]
that is $0$-connected and such that $I_{0}$ is finite and $H_{i}  \in \mathcal{F}$ for all $i \in I_0$. We can now inductively apply the procedure of Section \ref{sec: hard} to obtain for each $n\in \mathbb{N}$ a $G$-spectrum $X^n$ that fits into a stable cofiber sequence
\[     X^{n-1} \rightarrow  X^{n} \rightarrow \bigvee_{i\in I_{n}}\Sigma^n G/H_{i \ +} \rightarrow \Sigma  X^{n-1}\]
with $I_n$ finite and an $n$-connected map 
\[   f_n : X^n \rightarrow S^{0} \]
that extends $f_{n-1}$. Indeed, suppose $ X^{n-1}$ and $f_{n-1}$ have been constructed. Then we obtain an exact sequence
\[   0\rightarrow \ker d_{n-1} \rightarrow  \pi^{-}_{n-1}( X^{n-1}/ X^{n-2}) \xrightarrow{d_{n-1}} \ldots   \xrightarrow{d_2}\pi^{-}_{1}( X^{1}/ X^{0})  \xrightarrow{d_1}  \pi^{-}_{0}( X^{0})  \xrightarrow{d_0} \underline{A}\rightarrow 0 .      \]
Since there exists a free resolution of $\underline{A}$ in $\mathrm{Mack}_{\mathcal{F}}G$ consisting of finitely generated free modules, an application of Schanuel's lemma (e.g.~see \cite[Lemma VIII.4.4]{brown}) in this setting shows that $\ker d_{n-1} $ is finitely generated, i.e.~there exists a surjection of Mackey functors
\[    F= \bigoplus_{i \in I_{n}} \mathbb{Z}^G[-,H_i ] \rightarrow \ker d_n \]
where $I_n$ is finite and $H_i \in \mathcal{F}$ for all $i \in I_n$. This gives us an exact sequence of Mackey functors
\[   F \rightarrow  \pi^{-}_{n-1}( X^{n-1}/ X^{n-2}) \xrightarrow{d_{n-1}} \ldots   \xrightarrow{d_2}\pi^{-}_{1}( X^{1}/ X^{0})  \xrightarrow{d_1}  \pi^{-}_{0}( X^{0})  \xrightarrow{d_0} \underline{A}\rightarrow 0       \]
If we now apply the procedure of Section \ref{sec: hard} starting from (\ref{eq: freeres1}), one checks that we obtain the desired $G$-spectrum $X^n$. Since each map $f_n$ is $n$-connected, we have
\[     \mathrm{hocolim}_n X^n \cong S^0.   \]

In conclusion, we have constructed a stable model for $\underline{E}G$ that is of finite type. This shows that (2) implies (1), hence (1), (2),(3) and (4) are all equivalent. The final statement of the theorem follows from what is already proven and \cite[Th. 0.1]{LuckMeintrup}.
\end{proof}
\begin{remark} \rm
	The assumption of finite presentability above is necessary. Indeed, Bestvina and Brady (see \cite[Ex. 6.3.]{BestvinaBrady}) have famously constructed examples of torsion-free groups $G$ that are not finitely presented but do admit a finite length resolution of $\mathbb{Z}$ consisting of finitely generated free $\mathbb{Z}[G]$-modules. Such groups admit stable models for $\underline{E}G$ of finite type (even finite), but they do not admit an unstable finite type model for $\underline{E}G$ since that would imply that they are finitely presented. 
\end{remark}

Next, we determine when $S^0$ is a compact object in $\mathrm{Ho}(\Gspec)$. We start with the following easy and well-known lemma. Since the proof is short, we include it for completeness.
\begin{lemma}\label{lemma: compact} If $S^0$ is a compact object in $\mathrm{Ho}(\Gspec)$, then
	\[  \colim_n [S^{0}, X^n]^G \xrightarrow{\cong} [S^{0},\mathrm{hocolim}_n  X^n]^G   \]
	for any a sequence of maps
	 \[    X^0 \rightarrow  X^1 \rightarrow  X^2 \rightarrow \ldots \rightarrow  X^n \rightarrow \ldots  \]
	 in $\mathrm{Ho}(\Gspec)$. 
	
\end{lemma}
\begin{proof} Note that (by definition) the homotopy colimit fits into a stable cofiber sequence
	\[    \bigvee_{n\geq 0}  X^{n} \xrightarrow{1-\mathrm{Sh}} \bigvee_{n\geq 0}  X^{n} \rightarrow \mathrm{hocolim}_n  X^n \rightarrow  \bigvee_{n\geq 0}\Sigma X^n,   \]
	where $\mathrm{Sh}$ denotes the  shift-map.
Applying $[S^{0},-]^G$ and using the fact that $S^{0}$ is compact, we obtain a long exact sequence of abelian groups
\[   \bigoplus_{n\geq 0}[S^0, X^n ]^G \xrightarrow{1-\mathrm{Sh}}  \bigoplus_{n\geq 0}[S^0, X^n ]^G \rightarrow [S^{0},\mathrm{hocolim}_n  X^n  ]^G \rightarrow  \bigoplus_{n\geq 0}[S^0,\Sigma X^n ]^G \xrightarrow{1-\mathrm{Sh}}  \bigoplus_{n\geq 0}[S^0,\Sigma X^n ]^G.   \]
Since the maps $1-\mathrm{Sh}$ are injective and the cokernel of $\bigoplus_{n\geq 0}[S^0, X^n ]^G \xrightarrow{1-\mathrm{Sh}}  \bigoplus_{n\geq 0}[S^0, X^n ]^G$ is by definition  $\colim_n [S^{0}, X^n]^G$, we conclude that 
		\[  \colim_n [S^{0}, X^n]^G \xrightarrow{\cong} [S^{0},\mathrm{hocolim}_n  X^n]^G.   \]
\end{proof}

\begin{theorem}\label{prop: compact} Let $G$ be a countable discrete group and let $\mathcal{F}$ be its family of finite subgroups. The following are equivalent.
\begin{itemize}
	\item[(1)] $S^0$ is a compact object in $\mathrm{Ho}(\Gspec)$.
	\item[(2)] There exists a finite length resolution of the Burnside ring functor $\underline{A}$ in $\mathrm{Mack}_{\mathcal{F}}G$ consisting of finitely generated projective modules.
	\item[(3)] There exists a finite dimensional (stable) model for $\underline{E}G$ and there exists a finite type stable model for $\underline{E}G$.

\end{itemize}	
\end{theorem}
\begin{proof} Asssume that  $\Sigma^{\infty}\ul{E}G_{+}\cong S^0$ is a compact object and let $\{M_i\}_{i \in I}$ be a countable collection of Mackey functors for $G$. As noted in Section $3$, there exists a family of Eilenberg-Maclane $G$-spectra $\{HM_i\}_{i\in I}$ such that 
	\[      [\Sigma^{\infty}\ul{E}G_+, \Sigma^nHM_{i}]^G \cong \mathrm{H}^n_{\mathcal{F}}(G,M_i)  \]
	and 
	\[      [\Sigma^{\infty}\ul{E}G_+, \bigvee_{i\in I}\Sigma^nHM_i]^G \cong \mathrm{H}^n_{\mathcal{F}}(G,\bigoplus_{i \in I}M_i). \]
	for each $n\geq 0$.
	We conclude that for each $n\geq 0$, $\mathrm{H}^{n}_{\mathcal{F}}(G,-)$ commutes with countable direct sums of Mackey functors. By Proposition \ref{thm: bieri-eckmann} below in the setting of Mackey functors (note that $\mathcal{M}_{\mathcal{F}}G$ is countable if $G$ is), it follows that the Burnside ring $\underline{A}$ is of type $FP_{\infty}$ in the category of $G$-Mackey functors, meaning that there exists a (possibly infinite length) resolution $P_{\ast} \rightarrow \underline{A}$ consisting of finitely generated free $G$-Mackey functors. By Theorem \ref{th: finite type} this implies that there exist a stable model $ X$ for $\underline{E}G$ of finite type. In particular, $S^0\cong \mathrm{hocolim}_n  X^n$, where each $ X^n$ is a compact object of $\mathrm{Ho}(\Gspec)$. It follows from Lemma \ref{lemma: compact} that 
		\[  \colim_n [S^{0}, X^n]^G \xrightarrow{\cong} [S^{0},S^0]^G.   \]
By considering the identity map in $[S^{0},S^0]^G$, we conclude that there exists an $n \in \mathbb{N}$ such that $S^{0}$ is a retract of $ X^n$ in $\mathrm{Ho}(\Gspec)$. But then $\mathrm{H}^{n+1}_{\mathcal{F}}(G,M)=\mathrm{H}^{n+1}_{G}(S^0,M)$ is a retract of the $\mathrm{H}^{n+1}_{G}( X^n,M)$, which is zero for every Mackey functor $M$, hence $\underline{\mathrm{cd}}_{\mathcal{M}}(G)\leq n $. This implies that the kernel of $P_{n-1} \rightarrow P_{n-2}$ is a finitely generated projective module, proving that (1) implies (2).

Now assume that (2) holds. Then $\underline{\mathrm{cd}}_{\mathcal{M}}(G)<\infty$ and a application of Schanuel's lemma shows that there exists a resolution of $\underline{A}$ in $\mathrm{Mack}_{\mathcal{F}}G$ consisting of finitely generated free modules. It follows from Theorem \ref{th: finite type} that there exists a finite type stable model for $\underline{E}G$ and that the number of conjugacy class of finite subgroups of $G$ is finite. This implies that the length $l(G)$ of $G$ is also finite. Therefore the inequality (\ref{eq: length bound}) implies that $\underline{\mathrm{cd}}(G)<\infty$. We conclude that there exists a finite dimensional model for $\underline{E}G$ and hence also a finite dimensional stable model for $\underline{E}G$. This proves that (2) implies (3).

Next, assume that (3) holds. Let $ Y$ be an $n$-dimensional stable model for $\underline{E}G$ and let $ X$ be a finite type stable model for $\underline{E}G$. So in $\mathrm{Ho}(\Gspec)$ there are isomorphisms $g:  X \xrightarrow{\cong} S^{0}$ and $f:  Y \xrightarrow{\cong} S^{0} $ and a map $i:  X^n \rightarrow  X$. Denote $\alpha = f^{-1}\circ g \circ i $ and note that by cellular approximation, there exists a map $\beta:  Y \rightarrow  X^n$ such that $i\circ \beta = g^{-1}\circ f$.  It follows that $\alpha \circ \beta =\mathrm{Id}$, proving that $Y$, and hence also $S^{0}$, is a retract of $ X^n$. Since $X^n$ is compact and retracts of compact object are compact, we conclude that $S^{0}$ is compact. This show that (3) implies (1).

\end{proof}

Note that Theorem \ref{th: main theorem}, Theorem \ref{th: finite type} and Theorem \ref{prop: compact} together imply Theorem \ref{th: intro2} from the introduction. We now turn to finite stable models.
\begin{theorem} \label{th: finite} Let $G$ be a discrete group and let $\mathcal{F}$ be its family of finite subgroups. The following are equivalent.
\begin{itemize}
\item[(1)]  There exists a finite stable model for $\underline{E}G$ .
\item[(2)] There exists a finite stable model for $\underline{E}G$ of dimension $\underline{\mathrm{cd}}_{\mathcal{M}}(G)$.
\item[(3)] There exists a finite length resolution of $\underline{A}$ in $\mathrm{Mack}_{\mathcal{F}}G$ consisting of finitely generated free modules.

\end{itemize}

\end{theorem}

\begin{proof}
The proof is very similar to the proof of Theorem \ref{th: finite type}, so we will only give a sketch. Assume there exists a finite stable model  for $\underline{E}G$. Then its associated resolution is obviously a finite length free resolution of $\underline{A}$ in $\mathrm{Mack}_{\mathcal{F}}G$ consisting of finitely generated free modules, proving that (1) implies (3). Assuming (3) and $\underline{\mathrm{cd}}_{\mathcal{M}}(G)=n$, we obtain  $\mathcal{M}_{\mathcal{F}}G$-resolutions

\[   0 \rightarrow F_m \rightarrow F_{m-1} \rightarrow F_{m-2} \rightarrow \ldots \rightarrow F_{0} \rightarrow \underline{A} \rightarrow 0   \]
and
\[   0 \rightarrow P \rightarrow F_{n-1} \rightarrow F_{n-2} \rightarrow \ldots \rightarrow F_{0} \rightarrow \underline{A} \rightarrow 0   \]
with $m\geq n$  and	where each $F_i$ is free and finitely generated and $P$ is finitely generated projective. Another application of Schanuel's lemma shows that there exists a finitely generated free Mackey functor $F$ such that $P\oplus F = \widetilde{F}_n$ is finitely generated and free, i.e.~$P$ is stably free. This implies the existence of free resolution of $\underline{A}$ in $\mathrm{Mack}_{\mathcal{F}}G$ of length $n$, consisting of finitely generated free modules. This resolution can be used to build a finite stable model for $\underline{E}G$ of dimension $n$. This shows that (3) implies (2). Since (2) clearly implies (1), the theorem is proven.
	
\end{proof}

In the remainder of this section we will prove a refinement of the Bieri-Eckmann criterion for group cohomology (see \cite{BieriEck}) in the more general setting of modules over a category (see \cite[Th. 2.6]{StGreen} and \cite[Th. 5.3-5.4]{MPN2}). This criterion was used in the proof of Theorem \ref{prop: compact}. The refinement refers to the fact that we only assume preservation of countable direct sums, while the usual criterion asks for preservation of all filtered colimits. The proof is completely standard, but since  we could not find a reference, we have included it here.

 Assume that $\mathcal{C}$ is a countable category enriched in abelian groups, meaning that $\mathcal{C}$ is a small category such that the set of objects and all the abelian groups of morphisms of $\mathcal{C}$ are countable. A module $M$ over a $\mathcal{C}$ (i.e.~an additive contravariant functor from $\mathcal{C}$ to $\zmod$) is called countable is $M(c)$ is countable for every $c \in \mathrm{Ob}(\mathcal{C})$. Note that finitely generated free modules over $\mathcal{C}$ are countable. Since finitely generated modules over $\mathcal{C}$ are by definition quotients of finitely generated free modules, all finitely generated modules over $\mathcal{C}$ are countable. 
\begin{lemma}\label{lemma: fin gen} Let $M$ be a countable right module over $\mathcal{C}$. 
	\begin{itemize}
		\item[(1)] The module $M$ has a countable ascending union of finitely generated submodules
		\[   M_0 \subseteq M_1 \subseteq M_2 \subseteq \ldots M_{n} \subseteq M_{n+1} \subseteq \ldots   \]
		such that $\colim_n M_n=M$.
		\item[(2)] The module $M$ is finitely generated if and only if $\mathrm{Hom}_{\mathcal{C}}(M,-)$ preserves countable direct sums.
	\end{itemize}
\end{lemma}
\begin{proof} First enumerate the objects of $\mathcal{C}$, i.e.~write $\mathrm{Ob}(\mathcal{C})=\{c_0,c_1,c_2,\ldots\}$.
	Since $M(c_i)$ is countable for every $i\geq 0$, there exists for every object $c_i$ a $\mathcal{C}$-module map 
	
	\[  f_{i}: \bigoplus_{j \geq 0} \mathcal{C}(-,c_i) \rightarrow M \]	
	such that $f_i(c_i)$ is surjective.  For each pair $(i,j) \in \mathbb{N}\times \mathbb{N}$, let $M_i^j$ be the image of $j$-th summand $\mathcal{C}(-,c_i)$ under $f_i$. Now define for each $n\geq 0$, the submodule $M_n$ of $M$ generated by the $M_{i}^j$ for all $i\leq n$ and $j\leq n$. Then each $M_n$ is finitely generated and there is a countable ascending union 
	\[   M_0 \subseteq M_1 \subseteq M_2 \subseteq \ldots M_{n} \subseteq M_{n+1} \subseteq \ldots   \]
	such that $\colim_n M_n=M$.\\
	
	Now assume that $M$ is finitely generated. By definition, this means that there exists a surjection
	\[   \pi: F=\bigoplus_{i=0}^k \mathcal{C}(-,c_i) \rightarrow M.  \]
	It follows from Yoneda's lemma that $\mathrm{Hom}_{\mathcal{C}}(F,-)$ preserves countable direct sums. Let $K$ denote the kernel of $\pi$ and let $\{V_i\}_{i \in I}$ be a countable collection of modules over $\mathcal{C}$. Note that for every module $V$, the canonical map
	\[   \bigoplus_{i \in I} \mathrm{Hom}_{\mathcal{C}}(V,V_i) \rightarrow \mathrm{Hom}_{\mathcal{C}}(V,\bigoplus_{i\in I}V_i) \]
	is injective. We therefore obtain a commutative diagram with exact rows
	\[ \xymatrix{0 \ar[r] & \mathrm{Hom}_{\mathcal{C}}(M,\bigoplus_{i \in I}V_i) \ar[r] &  \mathrm{Hom}_{\mathcal{C}}(F,\bigoplus_{i \in I}V_i) \ar[r] &  \mathrm{Hom}_{\mathcal{C}}(K,\bigoplus_{i \in I}V_i) \\
		0 \ar[r] & \bigoplus_{i \in I}\mathrm{Hom}_{\mathcal{C}}(M,V_i) \ar[u]\ar[r] &  \bigoplus_{i \in I}\mathrm{Hom}_{\mathcal{C}}(F,V_i) \ar[r]\ar[u] &  \bigoplus_{i \in I}\mathrm{Hom}_{\mathcal{C}}(K,V_i)\ar[u]}\] 
	where the middle vertical arrow is an isomorphism and the right vertical arrow is injective. It follows from a diagram chase that
	\[   \bigoplus_{i \in I} \mathrm{Hom}_{\mathcal{C}}(M,V_i) \rightarrow \mathrm{Hom}_{\mathcal{C}}(M,\bigoplus_{i\in I}V_i) \]
	is an isomorphism, as desired.
	
	Finally, suppose $M$ is a countable module such that $\mathrm{Hom}_{\mathcal{C}}(M,-)$ preserves countable direct sums. By the first part of the lemma, we can write $M$ as a countable ascending union $\bigcup_{n\geq 0}M_n$ of finitely generated submodules. Now consider the directed system $\{M/M_{n}\}_{n\geq 0}$ and note that $\colim_n M/M_n = 0$ fits into a short exact sequence
	\[   0 \rightarrow \bigoplus_{n\geq 0}M/M_{n} \xrightarrow{\mathrm{Id}-\mathrm{Sh}}  \bigoplus_{n\geq 0}M/M_{n} \rightarrow \colim_n M/M_n\rightarrow 0.  \]
	Since $F(-)=\mathrm{Hom}_{\mathcal{C}}(M,-)$ preserves countable direct sums, we obtain an exact sequence
	\[   0 \rightarrow \bigoplus_{n\geq 0}F(M/M_{n}) \xrightarrow{\mathrm{Id}-F(\mathrm{Sh})}  \bigoplus_{n\geq 0}F(M/M_{n}) \rightarrow  F(\colim_n M/M_n)=0.  \]
	It follows that the cokernel of the map $\mathrm{Id}-F(\mathrm{Sh})$ is the colimit of 
	\[      F(M_1) \rightarrow F(M_2) \rightarrow  F(M_3) \rightarrow \ldots    \]
	and conclude that 
	\[ \colim_n \mathrm{Hom}_{\mathcal{C}}(M,M/M_n)=0.  \]
	This means that the image of the identity map $M \rightarrow M$ must be contained in $M_n$ for some $n$, implying that $M=M_n$ is finitely generated. 	
\end{proof}

\begin{proposition} \label{thm: bieri-eckmann}Let $\mathcal{C}$ be a countable category and $M \in \mathrm{Mod-}\mathcal{C}$ a countable module. For every $n\geq 0$, the following two statements are equivalent.
	\begin{itemize}
		\item[(1)] The functor $\mathrm{Ext}^k_{\mathcal{C}}(M,-)$ commutes with countable direct sums for every $k\leq n$.
		\item[(2)] There exists an exact sequence
		\[  P_n \rightarrow P_{n-1} \rightarrow \ldots \rightarrow P_0 \rightarrow M \rightarrow 0   \]
		in  $\mathrm{Mod-}\mathcal{C}$ such that all $P_i$ are finitely generated projective (free) modules.
		
	\end{itemize}
	
\end{proposition}

\begin{proof} Let us first prove by induction on $n$ that (2) implies (1). If $n=0$, then $(2)$ says that $M$ is finitely generated, in which case it follows from Lemma \ref{lemma: fin gen} that $\mathrm{Ext}^0_{\mathcal{C}}(M,-)=\mathrm{Hom}_{\mathcal{C}}(M,-)$ commutes with countable direct sums. Now let $n>0$ and assume there exists a projective resolution 
	
	\[\ldots \rightarrow P_n \rightarrow P_{n-1} \rightarrow \ldots \rightarrow P_0 \rightarrow M \rightarrow 0   \]
	such that $P_i$ is finitely generated for $i \in \{0,\ldots,n\}$. Let $\{V_i\}_{i \in I}$ be a countable collection of modules over $\mathcal{C}$.
	By induction we just need to show that $\mathrm{Ext}^n_{\mathcal{C}}(M,-)$ preserves countable direct sums. This follows easily from the commutative diagram
	\[\xymatrix{\mathrm{Hom}_{\mathcal{C}}(P_{n-1},\bigoplus_{i\in I}V_i) \ar[r] & 	\mathrm{Hom}_{\mathcal{C}}(P_n,\bigoplus_{i\in I}V_i) \ar[r] & 	\mathrm{Hom}_{\mathcal{C}}(P_{n+1},\bigoplus_{i\in I}V_i) \\
		\bigoplus_{i\in I}\mathrm{Hom}_{\mathcal{C}}(P_{n-1},V_i) \ar[r] \ar[u]& 	\bigoplus_{i\in I}\mathrm{Hom}_{\mathcal{C}}(P_n,V_i) \ar[r]\ar[u] & 	\bigoplus_{i\in I}\mathrm{Hom}_{\mathcal{C}}(P_{n+1},V_i),\ar[u]}\]
	where the left and middle vertical arrows are isomorphisms and the right vertical arrow is injective.

	Next we prove by induction on $n$ that (1) implies (2).	The case $n=0$ follows from Lemma \ref{lemma: fin gen}. Now let $n\geq 1$ and proceed by induction, assuming the theorem is true for all $k<n$. Assume that the Ext-functor $\mathrm{Ext}^k_{\mathcal{C}}(M,-)$ commutes with countable direct sums for all $k\leq n$. By the induction hypothesis we can construct an exact sequence
	\[ 0 \rightarrow K \rightarrow P_{n-1} \rightarrow \ldots \rightarrow P_0 \rightarrow M \rightarrow 0.    \]
	where all $P_i$ are finitely generated projective (free). By dimension shifting, there is a natural isomorphism
	\[   \mathrm{Ext}^n_{\mathcal{C}}(M,-)\cong  \mathrm{Ext}^1_{\mathcal{C}}(\tilde{K},-),   \]
	and a natural exact sequence
	\[   \mathrm{Hom}_\mathcal{C}(P_{n-1},-) \rightarrow  \mathrm{Hom}_\mathcal{C}(K,-) \rightarrow \mathrm{Ext}^1_{\mathcal{C}}(\tilde{K},-)\rightarrow 0 .  \]
	Here $\tilde{K}$ is the kernel of $P_{n-2} \rightarrow P_{n-3}$. Since $P_{n-1}$ is finitely generated, $\mathrm{Hom}_\mathcal{C}(P_{n-1},-)$ preserves countable direct sums by Lemma \ref{lemma: fin gen}. Since $\mathrm{Ext}^n_{\mathcal{C}}(M,-)\cong  \mathrm{Ext}^1_{\mathcal{C}}(\tilde{K},-)$ also preserves countable direct sums by assumption, it follows that $ \mathrm{Hom}_\mathcal{C}(K,-)$ preserves countable direct sums. Indeed, the canonical map
	\[   \bigoplus_{n\geq 0}  \mathrm{Hom}_\mathcal{C}(K,V_n)\rightarrow  \mathrm{Hom}_\mathcal{C}(K,\bigoplus_{n\geq 0}V_n)  \]
	is always injective and the commutative diagram with exact rows
	\[ \xymatrix{ \mathrm{Hom}_\mathcal{C}(P_{n-1},\bigoplus_{n\geq 0}V_n) \ar[r]  &  \mathrm{Hom}_\mathcal{C}(K,\bigoplus_{n\geq 0}V_n) \ar[r]  &  \mathrm{Ext}^1_{\mathcal{C}}(\tilde{K},\bigoplus_{n\geq 0}V_n) \ar[r]  & 0 \\
		\bigoplus_{n\geq 0} \mathrm{Hom}_\mathcal{C}(P_{n-1},V_n) \ar[r] \ar[u]^{\cong} & 	\bigoplus_{n\geq 0} \mathrm{Hom}_\mathcal{C}(K,V_n) \ar[r] \ar[u] &  \bigoplus_{n \geq 0}\mathrm{Ext}^1_{\mathcal{C}}(\tilde{K},V_n) \ar[r] \ar[u]^{\cong} & 0}\]
	shows that it is also surjective. We therefore conclude from Lemma \ref{lemma: fin gen} that $K$ is finitely generated, i.e.~there exists a finitely generated projective (free) module $P_n$ that surjects onto $K$, yielding an exact sequence
	\[ P_n \rightarrow P_{n-1} \rightarrow \ldots \rightarrow P_0 \rightarrow M \rightarrow 0   \]
	as desired.

\end{proof}

\section{An Example} \label{sec: Examples}
Let $A_5$ be the alternating group on $5$ elements. We recall the construction of a $2$-dimensional acyclic $A_5$-CW-complex without global fixed point, due to Floyd and Richardson (see \cite{floyd}). The particular construction presented here is taken from \cite[Section 9 - Example 4]{LearyNucinkis}. Consider the $1$-skeleton of the simplex with $5$ vertices $\{1,\ldots,5\}$. Recall that the conjugacy class of the $5$-cycle $(1,2,3,4,5)$ in $A_5$ contains $12$ elements, and that each element $x$ of order $5$ in $A_5$ is conjugate to $x^4$ but not to $x^2$ and $x^3$. Therefore, the conjugacy class of $(1,2,3,4,5)$ is of the form $\{x_1,x_1^{-1},x_2^{-1},\ldots,x_6,x_6^{-1}\}$. Now attach $6$ pentagonal cells to the aforementioned $1$-skeleton according to $x_1,x_2,\ldots,x_6$ to form the $2$-complex $M$. One can check that $M$ may also be obtained by identifying opposite faces of a dodecahedron by a twist of $\frac{\pi}{5}$. From this description it is easily seen to be a 2-spine of the punctured Poincar\'{e} homology $3$-sphere which proves that $M$ is acyclic. However, $M$ is not contractible since its fundamental group is isomorphic to $\mathrm{SL}_2(\mathbb{F}_5)$. In fact, it follows from \cite[Prop. 5]{BradyLearyNucinkis} that $M$ cannot be embedded in any contractible $2$-complex. Now let $L$ be the barycentric subdivision of $M$. Then $L$ is a $2$-dimensional acyclic flag complex admitting an admissible $A_5$-action such that $L^{A_5}$ is empty. Here, admissible means that simplices are fixed if and only if they are fixed pointwise. Since every proper subgroup of $A_5$ is solvable, it follows from a result of Segev (e.g. see \cite[Th 3.1.]{Adem}) that $L^H$ is acyclic for every proper subgroup of $A_5$. Denoting the family of proper subgroups of $A_5$ by $\mathcal{P}$, one checks that (e.g. see \cite[Example 5.1]{Adem}) the cellular chain complexes of $L^{H}$, for all $H \in \mathcal{P}$, assemble to form an exact $\mathcal{O}_{\mathcal{P}}A_5$-resolution of the constant functor $\underline{\mathbb{Z}}$ of the form
\begin{equation}   \label{eq: Lforafive}     0 \rightarrow \mathbb{Z}[-,A_5/e] \rightarrow \begin{array}{c}  
\mathbb{Z}[-,A_5/C_2]  \\  
\oplus \\
\mathbb{Z}[-,A_5/C_2]   \\
\oplus \\
\mathbb{Z}[-,A_5/C_3]     \end{array}  \rightarrow \begin{array}{c}  
\mathbb{Z}[-,A_5/A_4]  \\  
\oplus \\
\mathbb{Z}[-,A_5/D_5]   \\
\oplus \\
\mathbb{Z}[-,A_5/D_{3}]     \end{array} \rightarrow \underline{\mathbb{Z}} \rightarrow 
0 , \end{equation}
where $D_n$ is the dihedral group of order $2n$ and $C_n$ is the cyclic group of order $n$.
Denote the $1$-skeleton of $L$ by $L^{1}$. 

The space $L^{1}$ is a finite graph whose vertex set is denoted by $S$ and whose set of edges is denote by $E(L)$. The right angled Coxeter group $W$ associated associated to $L$ is the group defined by the presentation
\[       W = \langle S \ | \ s^2 \ \mbox{for all $s \in S$ and \ }   (st)^2 \ \mbox{if $(s,t) \in E(L)$} \rangle .  \]
Note that $W$ fits into the short exact sequence 
\[   1 \rightarrow N \rightarrow W \xrightarrow{p} F = \bigoplus_{s\in S} C_2 \rightarrow 1\]
where $p$ takes $s\in S$ to the generator of the $C_2$-factor corresponding to $s$.
A subset $J \subseteq S$ is called spherical if the subgroup $W_J=\langle J \rangle$ is finite (and hence isomorphic to $\bigoplus_{s \in J} C_2$). Note that the empty subset of $J$ is spherical. We denote the poset of spherical subsets of $S$ ordered by inclusion by $\mathcal{S}$ and its geometric realization by $K$. Note the $K$ is the cone over the barycentric subdivision of $L$. If $J \in \mathcal{S}$, then $W_J$ is called a spherical subgroup of $W$, while a coset  $wW_J$ is called spherical coset. We denote the poset of spherical cosets, ordered by inclusion, by $W\mathcal{S}$. Note that $W$ acts on $W\mathcal{S}$ by left multiplication, preserving the ordering. The Davis complex $X$ of $W$ is the geometric realization of $W\mathcal{S}$. One sees that $X$ is a proper $3$-dimensional cocompact $W$-CW-complex with strict fundamental domain $K$. Since $X$ admits a complete CAT(0)-metric such that $W$ acts by isometries, it follows that $X$ is a $3$-dimensional cocompact model for $\underline{E}W$ (see \cite[Th. 12.1.1 \& Th. 12.3.4]{DavisBook}). A consequence of this fact is that every finite subgroup of $W$ is subconjugate to some spherical subgroup of $W$. This implies that the group $N$ defined above is torsion-free, proving that $W$ is virtually torsion-free. We refer the reader to \cite{DavisBook} for more details and information about Coxeter groups. 

The singular set $X_{\mathrm{sing}}$ is by definition the subcomplex of $X$ consisting of all cells with non-trivial stabilizers. In others words, it is the geometric realization of the subposet of $W\mathcal{S}$ consisting of cosets $wW_J$ with $J\neq \emptyset$. Note that $X_{\mathrm{sing}}$ is two dimensional. Proposition 4 in \cite{BradyLearyNucinkis} says that $X_{\mathrm{sing}}^K$ is acyclic for every $K \in \mathcal{F}$, $\mathrm{vcd}(W)=\underline{\mathrm{cd}}_{\mathcal{M}}(W)=\underline{\mathrm{cd}}(W)=2$ but $\underline{\mathrm{gd}}(W)=3$. (warning: $L$ is denoted by $K$ in \cite{BradyLearyNucinkis} .) In particular, $X_{\mathrm{sing}}$ is not contractible. (Note that by \cite[Remark 6.4]{LP} there in fact does not exist any contractible proper $W$-CW-complex.) However,  $\Sigma^{\infty}X_{\mathrm{sing}+}$ is a $2$-dimensional stable model for $\underline{E}W$ and $\underline{\mathrm{gd}}_{\mathrm{st}}(W)=2$. Hence, $X_{\mathrm{sing}}$ is an example of a space that is not a model for $\underline{E}W$, but its suspension $\Sigma^{\infty}X_{\mathrm{sing}+}$ is a stable model for $\underline{E}W$ (see Proposition \ref{prop: susp model}). 
 
The action of $A_5$ on $L$ induces a map $A_5 \rightarrow \mathrm{Aut}(W)$ that allows one to form the semi-direct product $\Gamma=W \rtimes A_5$. The action of $A_5$ on $L$ also allows one to extend the action of $W$ on $X$ to $\Gamma$ such that $X$ becomes a $3$-dimensional cocompact model for $\underline{E}\Gamma$ and hence $\underline{\mathrm{gd}}(\Gamma)=3$ (see \cite[Lemma 3.5 and Example 5.1]{LP}). Here, $A_5$ acts trivially on the vertex in $X$ corresponding to $W_{\emptyset} \in W\mathcal{S}$. Since $\Gamma$ contains $W$ as a finite index subgroup, we have $\mathrm{vcd}(\Gamma)=\underline{\mathrm{cd}}_{\mathcal{M}}(\Gamma)=2$. However, it is proven in \cite[Theorem 1.1 and Example 5.1]{LP} that $\underline{\mathrm{cd}}(\Gamma)=3$. We conclude that there exists a $2$-dimensional stable model for $\underline{E}\Gamma$, but that this model cannot be of the form $\Sigma^{\infty}X_{+}$, for some proper $\Gamma$-CW-complex $X$. We finish this paper by explaining how such a $2$-dimensional stable model can be constructed. To do this, we will avoid using the general result that $\mathrm{vcd}(G)=\underline{\mathrm{cd}}_{\mathcal{M}}(G)$ and instead prove this directly for $\Gamma$ using the following lemma, where  $\mathcal{M}_{\mathcal{P}}A_5$ denotes the Mackey category of $A_5$ for the family of proper subgroups $\mathcal{P}$. 
\begin{lemma} The Burnside functor $\underline{A}: \mathcal{M}_{\mathcal{P}}A_5 \rightarrow \zmod: F \mapsto A(F)$ is a projective right $\mathcal{M}_{\mathcal{P}}A_5$-module.

\end{lemma}

\begin{proof} Let $\mathrm{H}^{\ast}_{\mathcal{M}_{\mathcal{P}}}(A_5,-)=\mathrm{Ext}^{\ast}_{\mathcal{M}_{\mathcal{P}}A_5}(\underline{A},-)$. It follows from \cite[Cor. 21.4]{GreenleesMay} that there exists a positive non-zero integer $n(\mathcal{P})$ such that multiplication with $n(\mathcal{P})$ annihaltes $\mathrm{H}^{1}_{\mathcal{M}_{\mathcal{P}}}(A_5,M)$ for every $\mathcal{M}_{\mathcal{P}}A_5$-module $M$. It follows from \cite[Example 21.5(iii)]{GreenleesMay} that $n(\mathcal{P})=1$. Hence $\mathrm{H}^{1}_{\mathcal{M}_{\mathcal{P}}}(A_5,M)=0$ for every $\mathcal{M}_{\mathcal{P}}A_5$-module $M$, proving that $\underline{A}$ is projective (but not free).
\end{proof}
Following \cite[Th. 3.8]{MartinezNucinkis06}, one checks that by applying the induction functor associated to $\mathcal{O}_{\mathcal{P}}A_5 \rightarrow \mathcal{M}_{\mathcal{P}}A_5$ to the $\mathcal{O}_{\mathcal{P}}A_5$-chain complex of the barycentric subdivision of $L$, one obtains a free $\mathcal{M}_{\mathcal{P}}A_5$-resolution $F_{\ast}$ of $\underline{A}$ of the form 
\begin{equation}   \label{eq: Lforafive2}     0 \rightarrow \mathbb{Z}^{A_5}[-,e]^6 \xrightarrow{s} \begin{array}{c}  
\mathbb{Z}^{A_5}[-,e]^6  \\  
\oplus \\
\mathbb{Z}^{A_5}[-,C_2]^4   \\
\oplus \\
\mathbb{Z}^{A_5}[-,C_3]^2     \end{array}  \rightarrow 
F_0  \rightarrow \underline{A} \rightarrow 0.  \end{equation}
Since $\underline{A}$ is a projective $\mathcal{M}_{\mathcal{P}}A_5$-module, it follows that there exists a map of $\mathcal{M}_{\mathcal{P}}A_5$-modules
\[    r: \begin{array}{c}  
\mathbb{Z}^{A_5}[-,e]^6  \\  
\oplus \\
\mathbb{Z}^{A_5}[-,C_2]^4   \\
\oplus \\
\mathbb{Z}^{A_5}[-,C_3]^2     \end{array}\rightarrow \mathbb{Z}^{A_5}[-,e]^6     \]
such that $r\circ s =\mathrm{Id}$. Using the fact that these are free functors, it follows that $r$ and $s$ extend to maps of $\mathcal{M}_{\mathcal{F}}A_5$-modules such that $r\circ s=\mathrm{Id}$. Here the family of finite subgroup $\mathcal{F}$ of $A_5$ coincides with the family of all subgroups of $A_5$.

As mentioned above, the Davis complex $X$ of $W$ is a $3$-dimensional cocompact model for $\underline{E}\Gamma$. This implies that $\Sigma^{\infty}X_{+}$ is a 3-dimensional stable model for $\underline{E}\Gamma$ and that its chain complex $\mathrm{ind}_{\pi}(C_{\ast}(X^-))$ is a free $\mathcal{M}_{\mathcal{F}}\Gamma$-resolution of $\underline{A}$. Letting $X_{\mathrm{sing}}$ denote the singular set of $X$ with respect to the $W$-action, there is a short exact sequence of $\mathcal{M}_{\mathcal{F}}\Gamma$-chain complexes 
\begin{equation}\label{eq: resgamma} 0 \rightarrow \mathrm{ind}_{\pi}(C_{\ast}(X^{-}_{\mathrm{sing}})) \rightarrow \mathrm{ind}_{\pi}(C_{\ast}(X^{-})) \rightarrow D_{\ast} \rightarrow 0 \end{equation}
where $D_{\ast}$ is obtained by applying the induction functor associated to $\mathcal{M}_{\mathcal{F}}A_5 \rightarrow \mathcal{M}_{\mathcal{F}}\Gamma$ to (\ref{eq: Lforafive2}) and is therefore of the form
\begin{equation*}       0 \rightarrow D_3=\mathbb{Z}^{\Gamma}[-,\Gamma/e]^6 \rightarrow \begin{array}{c}  
\mathbb{Z}^{\Gamma}[-,\Gamma/e]^6  \\  
\oplus \\
D_2=\mathbb{Z}^{\Gamma}[-,\Gamma/C_2]^4   \\
\oplus \\
\mathbb{Z}^{\Gamma}[-,\Gamma/C_3]^2     \end{array}\rightarrow D_1 \rightarrow  D_0 \rightarrow 
0 . \end{equation*}

 The splitting $r$ above yields a splitting $\rho$ of $D_3 \rightarrow D_2 $ which in turn  (recall that $X_{\mathrm{sing}}$ is $2$-dimensional) leads to a splitting  $\mu $ of $d_3$ in the resolution 
  \[   0 \rightarrow \mathrm{ind}_{\pi}(C_3(X)) \xrightarrow{d_3} \mathrm{ind}_{\pi}(C_2(X)) \xrightarrow{d_2}  \mathrm{ind}_{\pi}(C_1(X)) \xrightarrow{d_1}  \mathrm{ind}_{\pi}(C_0(X)) \xrightarrow{d_0} \underline{A} \rightarrow 0.   \]
 Therefore, $\ker d_1$ is projective (proving that $\underline{\mathrm{cd}}_{\mathcal{M}}(\Gamma)=2$),
 \[\ker d_1 \oplus  \mathrm{ind}_{\pi}(C_3(X))  \cong  \mathrm{ind}_{\pi}(C_2(X)), \]
 and we can construct a free $\mathcal{M}_{\mathcal{F}}\Gamma$-resolution
\[       0 \rightarrow  \mathrm{ind}_{\pi}(C_2(X)) \xrightarrow{\small \Big(\begin{array}{c}  d_2 \\ \mu \end{array}\Big)} \begin{array}{c}
                                                                                                   \mathrm{ind}(C_1(X)) \\ 
                                                                                                                    \oplus \\
                                                                                                                                    \mathbb{Z}^{\Gamma}[-,\Gamma/e]^6  \end{array} \xrightarrow{(d_1,0)} 
                                                                                                                                                                                                  \mathrm{ind}(C_0(X))   \xrightarrow{d_0} \underline{A} \rightarrow 0 .\]
Applying the techniques of the previous section, this resolution corresponds to a $2$-dimensional stable model  $ Y^{2}$ for $\underline{E}\Gamma$ where 
\[     Y^{1}= \Sigma^{\infty} X^{1}_+ \vee \bigvee_{i=1}^6 \Sigma \Gamma/e_+\]
and the spectrum $ Y^2$ fits into the stable cofiber sequence
\[       Y^{1} \rightarrow  Y^2 \rightarrow \Sigma^{\infty}X^2_{\mathrm{sing}}/X^1_{\mathrm{sing}} \vee A \xrightarrow{\Sigma^{\infty}\alpha \vee \Sigma^{2}\rho} \Sigma  Y^{1}  \]
where 

\[  A=  \bigvee_{i=1}^6 \Sigma^{2}\Gamma/e_{_+}\vee \bigvee_{i=1}^4 \Sigma^{2}\Gamma/C_{2_+}\vee \bigvee_{j=1}^2 \Sigma^{2}\Gamma/C_{3_+}  \]
and $\alpha$ is the unstable attaching map fitting into the homotopy cofiber sequence

\[  {X^1_{\mathrm{sing}}}_+ \rightarrow {X^2_{\mathrm{sing}}}_+ \rightarrow X^2_{\mathrm{sing}}/X^1_{\mathrm{sing}}\xrightarrow{\alpha} \Sigma {X^1_{\mathrm{sing}}}_+.  \]
The spectrum $ Y^2$ is not of form $\Sigma^{\infty}Z_{+}$ for any proper $\Gamma$-CW-complex $Z$ because $\rho$ contains transfer maps and therefore only exists in the stable world. Indeed, there do not exist $\Gamma$-maps from the $\Gamma$-sets $\Gamma/C_2$ and $\Gamma/C_3$ to the $\Gamma$-set $\Gamma/e$.

\end{document}